\documentclass[reqno,11pt]{amsart}
\usepackage{amssymb, amsmath,latexsym,amsfonts,amsbsy, amsthm,mathtools}
\usepackage{rotating}
\usepackage{hyperref}
\usepackage{tikz}
\usepackage{calc}
\usepackage{cases}
\usepackage{tikz-cd}
\usepackage{mathabx}
\usepackage{tikz-3dplot}
\usetikzlibrary{calc}
\usetikzlibrary{decorations.markings,arrows}
\usetikzlibrary{shapes.geometric}
\usepackage{color}
\usepackage{extarrows}
\usepackage{graphicx}

\allowdisplaybreaks[4]
\let\pa=\partial

\let\f=\frac

\let\pa=\partial

\def\f{\frac}

\def\eqdefa{\buildrel\hbox{\footnotesize def}\over =}

\newcommand{\beq}{\begin{equation}}
	\newcommand{\eeq}{\end{equation}}
\newcommand{\bali}{\begin{aligned}}
	\newcommand{\eali}{\end{aligned}}
\newcommand{\ben}{\begin{eqnarray}}
	\newcommand{\een}{\end{eqnarray}}
\newcommand{\beno}{\begin{eqnarray*}}
	\newcommand{\eeno}{\end{eqnarray*}}

\makeatletter

\newcommand{\Rmnum}[1]{\expandafter\@slowromancap\romannumeral #1@}
\makeatother

\newtheorem{Def}{Definition}[section]
\newtheorem{thm}{Theorem}[section]
\newtheorem{lem}{Lemma}[section]
\newtheorem{rmk}{Remark}[section]

\newtheorem{prop}{Proposition}[section]

\newcommand{\uu}{\mathbf{u}}
\newcommand{\vv}{\mathbf{v}}

\newcommand{\CF}{\mathcal{F}}
\newcommand{\CG}{\mathcal{G}}

\newcommand{\CH}{\mathcal{H}}

\newcommand{\CL}{\mathcal{L}}

\newcommand{\BT}{\mathbb{T}}

\newcommand{\BR}{\mathbb{R}}
\newcommand{\BC}{\mathbb{C}}
\newcommand{\BN}{\mathbb{N}}

\newcommand{\dx}{{\rm d} {x}}

\newcommand{\dy}{{\rm d} {y}}
\newcommand{\dz}{{\rm d} {z}}

\DeclareMathOperator{\Span}{span}
\title[Traveling waves for the inhomogeneous Euler equations]{Traveling waves near shear flows for the inhomogeneous Euler equations with non-constant density}

\author[Q. Zhao]{Qi Zhao}
\address[Q. Zhao]{Department of Mathematics, New York University Abu Dhabi, Saadiyat Island, P.O. Box 129188, Abu Dhabi, United Arab Emirates.}
\email{qz2875@nyu.edu}

\author[W. Zhao]{Weiren Zhao}
\address[W. Zhao]{Department of Mathematics, New York University Abu Dhabi, Saadiyat Island, P.O. Box 129188, Abu Dhabi, United Arab Emirates.}
\email{zjzjzwr@126.com, wz19@nyu.edu}

\begin{document}
	
	\begin{abstract}
		We investigate the existence and nonexistence of traveling wave solutions near monotonic shear flows with non-constant background density for the two-dimensional inhomogeneous Euler equations in a finite channel. For any small $\tau>0$, first, we construct nontrivial traveling waves with velocity and density in $H^{5/2-\tau}$ and $H^{3/2-\tau}$, respectively, showing that inviscid damping fails at these regularities. Second, when the distorted Rayleigh operator has no eigenvalues, we prove that such traveling wave solutions cannot exist in higher regularity spaces ($H^{5/2+\tau}$ for velocity and $H^{3/2+\tau}$ for density).
	\end{abstract}
	
	\maketitle

	\section{Introduction}

	In this paper, we study the two-dimensional inhomogeneous incompressible Euler system in a finite channel $(x,y)\in \Omega=\BT_{2\pi} \times [-1, 1]$:
	\beq\label{incompressible Euler system}
	\left\{\bali
	\pa_t \rho + \uu \cdot \nabla \rho &= 0,\\
	\rho(\pa_t \uu + \uu \cdot \nabla \uu) + \nabla P &= 0,\\
	\nabla \cdot \uu &= 0,\\
	v(t,x,\pm 1)&=0. 
	\eali\right.\eeq
	Here, $\rho(t,x,y) > 0$ is the density, $\uu = \big(u(t,x,y) ,v(t,x,y)\big)$ is the velocity field. Although when \(\rho \equiv 1\), \eqref{incompressible Euler system} reduces to the homogeneous Euler equation, which is globally well-posed, the global well-posedness of the above system remains widely open. We refer to \cite{bae2020blow, danchin2010well, danchin2011well, CY2019Onsager, W2023Onsager, fanelli2025geometric} and references therein for some local well-posedness results and the blow-up criterion.

	In the infinite periodic channel \(\Omega = \mathbb{T}_{2\pi} \times \mathbb{R}\), a nontrivial global solution was obtained near Couette flow \((y,0)\) and a constant background density \cite{chen2025nonlinear}. Since the velocity field is unbounded, the global solution in \cite{chen2025nonlinear} does not satisfy the energy conservation law,
	\begin{align}\label{eq:energy-con}
		\frac{d}{dt}\int_{\Omega} \rho|\uu|^2(t,x,y)dxdy=0.
	\end{align}
	In \cite{zhao2025inviscid}, the second author of this paper obtained a nontrivial global solution near a general monotonic shear flow \((u(y),0)\) with a more general positive background density \(
	\rho(y)\) in the finite periodic channel considered here. The global solution in \cite{zhao2025inviscid} satisfies \eqref{eq:energy-con}. Both results exploit the mixing effect of the shear flow and establish an inviscid damping estimate for the velocity perturbation. In both papers, careful study of nonlinear interactions leads to a requirement of Gevrey regularity for the perturbation. 
	
	In 1970, Orr \cite{Orr1907} observed that near Couette flow, the perturbed velocity approaches a shear flow as time goes to infinity for an ideal fluid governed by the Euler equations. This phenomenon, known as inviscid damping, was later rigorously proved by Bedrossian and Masmoudi \cite{BM2015} (see also \cite{IonescuJia2020cmp}). A negative result in this direction was established by Lin and Zeng \cite{LinZeng2011}, who constructed a steady solution near Couette flow in the Sobolev spaces $H^{5/2-}$ (see also \cite{castro2023traveling}).
	
	In 1960, Case \cite{Case1960} predicted that inviscid damping holds for stable monotonic shear flows. We refer to \cite{Jia2020arma, Jia2020siam, WeiZhangZhao2018, zillinger2017linear} for recent linear results. A mathematically rigorous proof for the nonlinear system was subsequently provided by Ionescu and Jia \cite{IJ2020} and Masmoudi and Zhao \cite{MasmoudiZhao2020}. A negative result similar to that of \cite{LinZeng2011} was obtained in \cite{sinambela2023transition} for certain monotonic shear flows. Regarding non-monotonic shear flows, Bouchet and Morita \cite{BM2010} first predicted the linear inviscid damping and the vorticity depletion, which was later proved in \cite{WeiZhangZhao2019, WeiZhangZhao2020} for symmetric shear flows and Kolmogorov flow \cite{WeiZhangZhao2020}. Recently, Ionescu, Iyer, and Jia \cite{ionescuSameerJia2022linear} established linear inviscid damping for a class of stable non-monotonic shear flows with one non-degenerate critical point. Conversely, a negative result was presented in \cite{CTW2023Stationary}, which constructs stationary structures near Kolmogorov and Poiseuille flows.
	
	For the inhomogeneous, incompressible Euler equation \eqref{incompressible Euler system}, due to the influence of density---especially with a non-constant density profile---proving both positive and negative results is challenging. The objectives of this paper are threefold:
	\begin{itemize}
		\item[1.] To establish the existence of a non-trivial global solution to the system \eqref{incompressible Euler system}.
		\item[2.] To construct a traveling wave solution in the vicinity of any monotonic shear flow with a positive background density. This result serves as an analog to that of \cite{LinZeng2011}.
		\item[3.] To demonstrate that when the distorted Rayleigh operator $u(y) \text{Id} - ({u'}\rho)'(y) \tilde{\Delta}^{-1}$ has no eigenvalues or embedded eigenvalues, such traveling wave solutions cannot exist if the perturbations are sufficiently smooth. Here $(u(y),0)$ is the monotonic shear flow, and $\rho(y)$ is the positive background density. 
	\end{itemize}

	\subsection{Main results}
	In this paper, we construct traveling wave solutions near the general monotonic shear flow \((u(y),0)\) with a more general nonnegative density \(\rho(y)\). More precisely, we are looking for solutions near a steady state 
	\beq\label{expression of shear flow}
	\uu = (u(y), 0), \quad \rho = \rho(y), \quad P = \text{cons.}.
	\eeq
	and of the form
	\begin{align}\label{def: traveling-wave}
		\uu(t,x,y) = \uu_s(x-ct,y),\quad \rho(t,x,y)=\rho_s(x-ct,y).
	\end{align}

	Our first result states as follows:
	
	\begin{thm}\label{Thm: travel-wave}
		Let $u, \rho \in C^{  4}(-1, 1)$ satisfy
		$$u(0) = 0, \quad u'(y),\ \rho(y) > c_0, \text{ for some }c_0>0 \text{ and all }y \in (-1,1).$$
		For any $\tau > 0$, there exists $\epsilon_0 >0$ such that for all $0 <\epsilon < \epsilon_0$, there is a steady non-sheared flow $\uu_\epsilon(x, y)$ and $ \rho_\epsilon (x, y)$ (with $c = 0$) to Euler equations \eqref{incompressible Euler system} and 
		\beq\bali\nonumber
		\| \uu_{\epsilon} - (u(y),0)\|_{H^{\f{5}{2} - \tau}(\BT_{2\pi} \times (-1,1))} + \left\|\rho_{\epsilon} - \rho(y)\right\|_{H^{\f{3}{2} - \tau}(\BT_{2\pi} \times (-1,1))} \leq \epsilon.
		\eali\eeq
	\end{thm}
	Some remarks are listed in order: 
	\begin{itemize}
		\item 
		Unlike in \cite{sinambela2023transition}, we do not impose any structural assumptions on the shear flow other than monotonicity.
		
		\item    The assumption $u(0) = 0$ is to ensure the existence of steady solution, namely, $c=0$ in \eqref{def: traveling-wave}. Without such an assumption, one can follow the same proof and obtain the traveling wave solution \eqref{def: traveling-wave} with $c=u(0)$. 
		
		\item    The above theorem immediately implies that nonlinear inviscid damping
		is not true in any velocity space $H^{\f{5}{2} - \tau}$  or density space $H^{\f{3}{2} - \tau}$ neighborhood of monotonic shear flow and background density. 

        \item   {  The $C^4$ regularity assumption on $u$ and $\rho$ in Theorem~\ref{Thm: travel-wave} is not essential for the existence result. It is used in the two-parameter construction to retain $C^3$ regularity after introducing the density perturbation $\theta_{l,\eta}=\theta+l\eta u'\digamma(y/\eta)$, which involves one derivative of $u$. By setting $l=0$, thereby omitting this additional density perturbation, and adjusting the spectrum through $m$ alone, the same existence conclusion holds under the weaker assumption $u,\rho\in C^3(-1,1)$.
        }
	\end{itemize}

	We then show that the value $(\f{5}{2},\f32)$ is somehow optimal. Actually, our next theorem shows that there exists no nontrivial steady traveling flows in the velocity space $H^{\f{5}{2} + \tau}$  and density space $H^{\f{3}{2} + \tau}$ neighborhood of a monotonic flow.

	\begin{thm} \label{Theorem of nonlinear stability}
		Let $u, \rho \in C^3(-1, 1)$ be such that $u'(y), \rho(y) > c_0$ for some constant $c_0 > 0$. Let $\tilde{\Delta} = \rho \partial_x^2 + \partial_y \left( \rho \partial_y (\cdot) \right)$ be the distorted Laplace operator and assume that the distorted Rayleigh operator 
        $$u(y) \operatorname{Id} - (u'\rho)'(y) \tilde{\Delta}^{-1}$$ 
        has no eigenvalues or embedded eigenvalues.  For any $\tau >0$, there exists $\epsilon_0 >0$ such that any traveling wave solution $\uu_s(x-ct,y),  \rho_s(x-ct,y)$ to Euler equations \eqref{incompressible Euler system} satisfying
		\beq\bali\nonumber
		\| \uu_s - (u(y),0)\|_{H^{\f{5}{2} + \tau}(\BT_{2\pi} \times (-1,1))} + \left\| \rho_s - \rho(y) \right\|_{H^{\f{3}{2} + \tau}(\BT_{2\pi} \times (-1,1))} \leq \epsilon_0,
		\eali\eeq
		must be a shear flow.
	\end{thm}
	
	\begin{rmk}
		We note that while this absence of traveling waves is necessary for inviscid damping, it does not guarantee nonlinear asymptotic stability---which currently requires Gevrey regularity \cite{zhao2025inviscid}. It remains open even for the global existence when the perturbations are in any Sobolev spaces. 
		
		For lower regularity, Arnold's type stability is also open for inhomogeneous Euler equations \cite{abarbanel1986nonlinear}, which was well-studied for the homogeneous Euler equations \cite{T2010Variants, arnold2013conditions, GS2024Arnold, wang2025nonlinear}.  
	\end{rmk}

	\subsection{Idea of proof}\label{sebsection of idea of proof}
	
	By introducing $\theta_s = \rho_s^{-1}$, the vorticity $\omega_s = \nabla \times \mathbf{u}_s$, and the stream function $\Psi_s$, defined by
	$$u_s = - \frac{\partial \Psi_s}{\partial y}, \quad v_s =  \frac{\partial \Psi_s}{\partial x},$$
	the equations for the traveling wave solution $(\mathbf{u}_s, \theta_s)$ with speed $c = 0$ become:
	\beq\label{incompressible Euler system: change 1}
	\left\{\bali
	\nabla^\bot \Psi_s \cdot \nabla \theta_s&= 0,\\
	\vv_s \cdot \nabla \vv_s + \theta_s \nabla P &= 0.
	\eali\right.\eeq
	From \eqref{incompressible Euler system: change 1}, there exists a function $G$ such that
	\beq\label{expression of G}
	\theta_s = G(\Psi_s).
	\eeq
	Substituting this back into \eqref{incompressible Euler system: change 1}, we obtain the equivalent form
	\beq\nonumber
	\nabla^\bot \Psi_s \cdot \left(\omega_s - \f{G'(\Psi_s)}{2G(\Psi_s)} |\nabla \Psi_s|^2\right) =0.
	\eeq
	Thus, there exists a function $L$ such that
	\beq\label{Long's equation}
	\Delta \Psi_s - \f{G'(\Psi_s)}{2G(\Psi_s)} |\nabla \Psi_s|^2 = L(\Psi_s),
	\eeq
	which is also called the Long's equation \cite{abarbanel1986nonlinear, long1953some, yih2012stratified}. It is straightforward to verify that the shear flow \eqref{expression of shear flow} is a solution to \eqref{Long's equation}, where the associated functions $G$ and $L$ are determined by \eqref{expression of G} and \eqref{Long's equation} as
	\beq\bali\label{expression of origion G and L}
	G(\Psi) = \theta(y)\eqdefa \frac{1}{\rho(y)}, \quad L(\Psi) =  - u' + \f{ \theta'}{2 \theta} u,
	\eali\eeq
	Inspired by \cite{LinZeng2011, sinambela2023transition, WangZhangZhu2020}, our approach to generating a non-trivial branch of steady solutions relies on the Crandall-Rabinowitz local bifurcation theorem \cite{CR1971}. The structure of these solutions often resembles the Kelvin-Helmholtz instability patterns. 
	
	The main difficulties concern the well-posedness of Long's equation \eqref{Long's equation}, specifically the regularity of the functions $G$ and $L$ defined in \eqref{expression of origion G and L}, as well as the characterization of the kernel of the linearized operator associated with \eqref{Long's equation}, which is crucial for applying the bifurcation theorem. Note that the stream function $\Psi$ is monotonic on both $(-1,0)$ and $(0,1)$, which necessitates defining $G$ and $L$ piecewise on these intervals. To ensure consistency at $y = 0$, there must exists $\delta > 0$ such that
	$$G(\Psi \big|_{- \delta < y < 0}) = \theta \big|_{- \delta < y < 0}, \quad G(\Psi \big|_{ 0 < y < \delta}) = \theta \big|_{ 0 < y < \delta }.$$
	This implies that $\theta$ in a negative neighborhood of $y = 0$ is uniquely determined by $\Psi$ and $\theta$ in the positive neighborhood:
	\beq\label{condition 1 from G}
	\theta \big|_{- \delta < y < 0} =  \theta \big|_{ 0 < y < \delta } \circ \left( \Psi \big|_{ 0 < y < \delta} \right)^{-1} \circ \Psi \big|_{- \delta < y < 0}.
	\eeq
	Similarly, the well-posedness of $L$ require that
	\beq\label{condition 1 from F}
	\left( - u' + \f{ \theta' }{2 \theta} u \right) \big|_{- \delta < y < 0} =  \left( - u' + \f{ \theta' }{2 \theta} u \right) \big|_{ 0 < y < \delta } \circ \left( \Psi \big|_{ 0 < y < \delta} \right)^{-1} \circ \Psi \big|_{- \delta < y < 0}.
	\eeq
	Furthermore, to ensure the regularity of $G$ and $L$ near $\Psi = 0$, the condition $\Psi'(0) = 0$ implies
	\beq\label{condition 2 from G and F}
	\theta'(0) = 0, \quad \left( - u' + \f{ \theta' }{2 \theta} u \right)'(0) = 0.
	\eeq
	A direct approach to satisfying conditions \eqref{condition 1 from G}--\eqref{condition 2 from G and F} is to introduce the modified profiles
	\beq\bali\label{eq:u,theta^delta}
	u^\delta =& \chi_\delta(y) u'(0)y + \left(1 - \chi_\delta(y)\right) u(y),\\
	\theta^\delta =& \chi_\delta(y) \theta(0) + \left(1 - \chi_\delta(y)\right) \theta(y),
	\eali\eeq
	where $\chi_\delta$ is a smooth cut-off function supported on $(-2\delta, 2\delta)$ with $\chi_\delta \equiv 1$ on $(-\delta, \delta)$.
	
	To study the linearized operator of \eqref{Long's equation}, we compute the Fr\'echet derivative of \eqref{Long's equation} at $\Psi$, which is given by
	\beq\bali\nonumber
	\CG \psi :&=\f{\delta}{\delta \Psi_s} \left(\Delta \Psi_s - \f{G'(\Psi_s)}{2G(\Psi_s)} |\nabla \Psi_s|^2 - L(\Psi_s)\right) \bigg|_\Psi \psi \\
	&= \Delta \psi - \f{G'(\Psi)}{G(\Psi)} \Psi' \pa_y \psi - \left(\f{G'}{2G} \right)'(\Psi) |\Psi'|^2 \psi - L'(\Psi) \psi\\
	&= \theta \tilde{\Delta} \psi - \f{\theta}{u} \left( \f{u'}{\theta} \right)' \psi.
	\eali\eeq
    Notice that $\CG\psi=\frac{1}{u \rho}(u(y)\mathrm{Id}-(u'\rho)'\tilde{\Delta}^{-1})\tilde{\Delta}\psi$. To find the kernel of $\CG$, we need to study the homogeneous distorted Rayleigh equation
	\beq\bali\nonumber
	\theta \f{{\rm d}}{{\rm d}y} \left( \f{1}{\theta} \f{{\rm d}}{{\rm d}y} \hat{\psi} \right) -  \f{\theta}{u} \left( \f{u'}{\theta} \right)' \hat{\psi} - k^2 \hat{\psi} = 0,
    \eali\eeq
	with the Dirichlet boundary condition $\psi(1) = \psi(-1) = 0$. It is equivalent to show that $-k^2$ is an eigenvalue of the following operator
	\beq\nonumber
	\CH_{u, \theta} = - \theta \f{{\rm d}}{{\rm d}y} \left( \f{1}{\theta} \f{{\rm d}}{{\rm d}y} \cdot \right) + \f{ \theta }{ u} \left( \f{ u' }{ \theta } \right)',
	\eeq
	with the Dirichlet boundary. However, for general $(u, \theta)$, the distorted Rayleigh operator may not have any eigenvalue (for example, $(\f{ u' }{ \theta } )'\geq 0$), namely, $\mathrm{Ker(\CG)=\{0\}}$ or $-k^2\notin \sigma(\CH_{u, \theta})$. Thus, to create a designed eigenvalue of $\mathcal{H}_{u, \theta}$, we introduce the perturbed shear profiles are defined as follows:
	\beq\bali\nonumber
	u^\delta_{m, \gamma}(y) :=& u^\delta(y) + m \gamma \int_0^y \theta^\delta(y') \Gamma(\f{y'}{\gamma}) \dy' ,\\
	\theta^\delta_{l, \eta}(y) :=& \theta^\delta(y) + l \eta (u^\delta)'(y) \digamma(\f{y}{\eta}),
	\eali\eeq
	where $u^{\delta}, \theta^{\delta}$ are given in \eqref{eq:u,theta^delta} which ensure a local symmetric property, and the functions $\Gamma$ and $\digamma$ are required to satisfy conditions \eqref{eq: Gamma and digamma's condition}, which create the designed eigenvalue by taking $\eta$ and $\gamma$ sufficiently small. With the modified shear flow and density function, the associated linear operator becomes
	\beq\nonumber
	\CH^\delta_{m, \gamma, l, \eta} = - \theta^\delta_{l,\eta} \f{{\rm d}}{{\rm d}y} \left( \f{1}{\theta^\delta_{l,\eta}} \f{{\rm d}}{{\rm d}y} \cdot \right) + \f{\theta^\delta_{l,\eta}}{u^\delta_{m,\gamma}} \left( \f{(u^\delta_{m,\gamma})'}{\theta^\delta_{l,\eta}} \right)',
	\eeq
	By Sturm-Liouville theory, the operator $\mathcal{H}^\delta_{m, \gamma, l, \eta}$ has at most one negative eigenvalue, which corresponds to the minimum of the associated energy functional given by
	$$E_{m,\gamma,l,\eta}[\psi] =  \int_{-1}^1 \f{1}{\theta^\delta_{l,\eta}} |\f{d}{dy}  \psi|^2 
	+ \f{1}{u^\delta_{m,\gamma}} \left(\f{(u^\delta_{m,\gamma})'}{\theta^\delta_{l,\eta}} \right)' |\psi|^2 \dy. $$
	To avoid the complex influence of the functions $\Gamma$ and $\digamma$, we study the limiting behavior of $E_{m,\gamma,l,\eta}$ as $\gamma, \eta \rightarrow 0$, which can be expressed as
	$$E_{m,l}[\psi] = \int_{-1}^1 \f{1}{\theta^\delta} |\f{d}{dy} \psi|^2 + \f{1}{u^\delta}\left(\f{(u^\delta)'}{\theta^\delta}\right)'|\psi|^2 \dy - \f{\Lambda_{m, l}}{\theta^\delta(0)} |\psi(0)|^2,$$
	corresponding to the operator
	$$\CH^\delta_{m,l} = - \theta^\delta \f{{\rm d}}{{\rm d}y} \left( \f{1}{\theta^\delta} \f{{\rm d}}{{\rm d}y} \cdot \right) + \f{ \theta^\delta }{ u^\delta } \left( \f{ (u^\delta)' }{ \theta^\delta } \right)' - \Lambda_{m, l} \delta(0).$$
	Let $\lambda_{m,\gamma,l,\eta}$ and $\lambda_{m,l}$ denote the smallest eigenvalues of $\mathcal{H}^\delta_{m, \gamma, l, \eta}$ and $\mathcal{H}^\delta_{m,l}$, respectively. As shown in Lemma \ref{Lemma: Lambda(lambda)}, $\lambda_{m,l}$ is strictly decreasing with respect to the parameter $\Lambda_{m, l}$. Furthermore, for sufficiently small $\gamma$ and $\eta$, we obtain
	\beq\nonumber
	|\lambda_{m,l} - \lambda_{m,\gamma,l,\eta}| \leq C |m \gamma| + C|l \eta|,
	\eeq
	and
	\beq\nonumber
	\lambda_{m,\gamma,l,\eta} - \lambda_{\tilde{m},\gamma,\tilde{l},\eta} 
	\leq - c \Lambda(m - \tilde{m}, l - \tilde{l}) + {  C (\gamma + \eta) \left(|m - \tilde{m}| + |l - \tilde{l}| \right)}.
	\eeq
	These estimates imply that $\lambda_{m, \gamma, l, \eta}$ is strictly decreasing with respect to $m$ and $l$. This monotonicity allows us to determine parameters $m$ and $l$ for any desired value $\lambda_{m,\gamma,l,\eta} = -k^2$.

	\subsubsection{Nonexistence of the traveling wave solution in high regularity}
	Let $\uu, \theta$ of the form \eqref{def: traveling-wave} be a traveling wave solution to the Euler equations \eqref{incompressible Euler system: change 1}. Then $u_s, v_s$ and $ \theta_s$ satisfy the equations:
	\beq\nonumber
	\left\{ \bali(u_s - c) \pa_x \theta_s + v_s \pa_y \theta_s= 0,\\
	\f{1}{\theta_s} \left((u_s - c) \pa_x \uu_s + v_s \pa_y \uu_s\right) + \nabla P = 0.
	\eali\right.
	\eeq
	Applying the curl to the second equation, and invoking the first one together with the incompressibility condition $\partial_x u_s + \partial_y v_s = 0$, yields
	\beq\label{eq0.2}
	(u_s - c) \nabla \cdot \left( \f{1}{\theta_s} \nabla v_s \right) - v_s \nabla \cdot \left( \f{1}{\theta_s} \nabla u_s \right) = 0,
	\eeq
    which can be rewritten as
    \beq\label{eq: 1}
    (u_s-c_{\epsilon})\tilde{\Delta} v_s-\left( \f{u'}{\theta} \right)' v_s=v_s \nabla \cdot \left( \f{1}{\theta_s} \nabla u_s \right)-\left( \f{u'}{\theta} \right)' v_s+(c-c_{\epsilon})\tilde{\Delta} v_s\eqdefa F. 
    \eeq
    Since the distorted Rayleigh operator has no eigenvalue and no embedded eigenvalue, we expect that the above equation has a solution. Indeed, by a contradiction argument, we prove that
    $$\|\nabla v_s\|_{L^2} \leq C \|F\|_{H^{\f{1}{2} + \tau}}$$
    for some constant $C$ independent of $|c-c_{\epsilon}|$. We refer to Lemma \ref{resolvent estimate} for more details. Note that the right-hand-side of \eqref{eq: 1} is of size $o(\epsilon_0^2)$, which implies $v_{\epsilon}\equiv 0$.

	\section{The existence of cat's eye structure} \label{section 2}
	
	This section is devoted to the proof of Theorem \ref{Thm: travel-wave}. 
	
	\subsection{Perturbation of the velocity and density}\label{Perturbations}
	
	In the first subsection, we introduce the perturbed velocity $u_{m, \gamma}$ and density $\theta_{l,\eta}$, and investigate the spectral properties of the associated perturbed distorted Rayleigh operator. Understanding these properties is crucial for analyzing the generation of cat's eye structures arising from shear flow. The primary objective is to deduce the spectral behavior of the limiting operator as $\gamma$ and $\eta$ tend to zero. As we demonstrate in Lemma \ref{Lemma: Lambda(lambda)}, the minimal eigenvalue depends monotonically on $m$ and $l$.
	
	Let us begin with a general monotonic shear flow $u(y), \theta(y)$. We define the perturbations as follows :
	\beq\bali\nonumber
	u_{m, \gamma}(y) :=& u(y) + m \gamma \int_0^y \theta(y') \Gamma(\f{y'}{\gamma}) \dy' := u(y) + \tilde{\Gamma}(y),\\
	\theta_{l, \eta}(y) :=& \theta(y) + l \eta u'(y) \digamma(\f{y}{\eta}) := \theta(y) + \tilde{\digamma}(y),
	\eali\eeq
	where $\Gamma,\digamma$ are {  smooth even functions satisfying $\Gamma',\digamma'\in L^1(\mathbb R)\cap H^{1/2}(\mathbb R) $ and }
	\beq\label{eq: Gamma and digamma's condition}
	\int_\BR \f{\Gamma'(y)}{y} \dy = -1, \quad \int_\BR \f{\digamma'(y)}{y} \dy = 1,
	\eeq
	
    {  In this subsection and Section~\ref{sec:The convergence of the perturbed}, we assume that the auxiliary profiles $u,\theta$ satisfy
    $$ u''(0)=\theta'(0)=0. $$
    In the proof of Theorem~\ref{Thm: travel-wave}, we apply this analysis to the preparatory profiles $u^\delta,\theta^\delta$, which satisfy these conditions by construction. Thus these conditions impose no additional structural restrictions on the original background profiles.}
	For convenience, we denote the perturbed, distorted Rayleigh operator by
	$$\CH_{m, \gamma, l, \eta} := \CH_{u_{m,\gamma}, \theta_{l,\eta}} = - \theta_{l,\eta} \pa_y \left( \f{1}{\theta_{l,\eta}} \pa_y \cdot \right) + \f{\theta_{l,\eta}}{u_{m,\gamma}} \left( \f{u_{m,\gamma}'}{\theta_{l,\eta}} \right)', $$
	and we define $\CH_{m, l}$ as the limiting operator obtained as $\gamma$ and $\eta$ tend to zero, in the sense that 
	$$\langle \CH_{m, \gamma, l, \eta}\, \psi, \phi \rangle \rightarrow \langle \CH_{m, l}\, \psi, \phi  \rangle, \quad \text{as } \gamma, \eta \rightarrow 0,$$
	for any $\psi, \phi \in H_0^1(-1, 1)$. Hereafter, $\langle \cdot, \cdot \rangle$ denotes the inner product in {  $L^2(-1, 1)$}. Under the condition \eqref{eq: Gamma and digamma's condition}, we obtain the following result:
	
	\begin{lem}
		
		For all $m, l \in \BR$, we have
		\beq\nonumber
		\CH_{m,l} = \CH_{u, \theta} - \left( \f{u'(0)l}{\theta(0)} + \f{\theta(0)m}{u'(0)} \right) \delta(0).
		\eeq
	\end{lem}
	\begin{proof}
		For any $\psi,  \phi \in H_0^1(-1, 1)$, we have
		\beq\bali\nonumber
		\langle  &\left( \CH_{m, \gamma, l, \eta} - \CH_{u, \theta} \right) \psi, \phi \rangle = \int_{-1}^{1} \left(\f{\theta_{l,\eta}'}{\theta_{l,\eta}} - \f{\theta'}{\theta}\right) \pa_y \psi \phi \,\dy\\
		&+ \int_{-1}^{1} \left(\f{\theta_{l,\eta}}{u_{m, \gamma}} \left(\f{u_{m, \gamma}'}{\theta_{l,\eta}}\right)' - \f{\theta}{u_{m, \gamma}} \left(\f{u_{m, \gamma}'}{\theta}\right)'\right)  \psi \phi \,\dy + \int_{-1}^{1}  \left( \f{\theta}{u_{m, \gamma}} \left(\f{u_{m, \gamma}'}{\theta}\right)' - \f{\theta}{u} \left(\f{u'}{\theta}\right)'\right)  \psi \phi \,\dy.
		\eali\eeq
		Letting  $\eta \to 0$, we obtain
		\beq\bali\nonumber
		\lim_{\eta \rightarrow 0} \int_{-1}^{1} \left(\f{\theta_{l,\eta}'}{\theta_{l,\eta}} - \f{\theta'}{\theta}\right) \pa_y \psi \phi \dy = 0,
		\eali\eeq
		and 
		\beq\bali\nonumber
		\lim_{\eta \rightarrow 0} \int_{-1}^{1} \left(\f{\theta_{l,\eta}}{u_{m, \gamma}} \left(\f{u_{m, \gamma}'}{\theta_{l,\eta}}\right)' - \f{\theta}{u_{m, \gamma}} \left(\f{u_{m, \gamma}'}{\theta}\right)'\right) \psi \phi \dy 
		= -\f{u'(0)l}{\theta(0)} \psi(0) \phi(0) \int_\BR \f{\digamma'(y)}{y} \dy.
		\eali\eeq
		Similarly, letting $\gamma \to 0$, we find that
		\beq\bali\nonumber
		\lim_{\gamma \rightarrow 0} \int_{-1}^{1} \left( \f{\theta}{u_{m, \gamma}} \left(\f{u_{m, \gamma}'}{\theta} \right)' - \f{\theta}{u} \left(\f{u'}{\theta}\right)'\right) \psi \phi \dy 
		= \f{\theta(0)m}{u'(0)} \psi(0) \phi(0) \int_\BR \f{\Gamma'(y)}{y} \dy.
		\eali\eeq
		Applying condition \eqref{eq: Gamma and digamma's condition}, we get
		\beq\nonumber
		\CH_{m,\gamma,l,\eta} \rightarrow \CH_{u, \theta} - \left( \f{u'(0)l}{\theta(0)} + \f{\theta(0)m}{u'(0)} \right) \delta(0), \quad \text{when } \gamma, \eta \rightarrow 0.
		\eeq
	\end{proof}
	
	The explicit form of $\CH_{m, l}$ allows us to investigate its spectrum. Recall that $\lambda$ is an eigenvalue of $\CH_{m, l}$ if the corresponding eigenvector $\psi$ satisfies the Rayleigh equation
	\beq\label{Rayleigh function of CH_{m,0,l,0}}
	- \theta \pa_y \left(\f{1}{\theta} \pa_y \psi\right) + \f{\theta}{u} \left(\f{u'}{\theta}\right)' \psi = \lambda \psi + \Lambda \psi(0) \delta(0).
	\eeq
	subject to the Dirichlet boundary condition $\psi(-1) = \psi(1) = 0.$ Here, we define $\Lambda = \Lambda_{m, l}$ as 
	\beq\label{definition of Lambda_m,l}
	\Lambda_{m, l} = \f{\theta(0)}{u'(0)} m + \f{u'(0)}{\theta(0)} l.
	\eeq
	The following lemma establishes the continuous and monotonic relationship between the eigenvalues of $\CH_{m, l}$ and the parameter $\Lambda$.
	
	\begin{lem} \label{Lemma: Lambda(lambda)}
		For any $\lambda \leq 0$, the Rayleigh equation \eqref{Rayleigh function of CH_{m,0,l,0}} admits a nontrivial solution for a unique $\Lambda(\lambda) \in \BR$. Moreover, the function $\Lambda(\lambda)$ is strictly decreasing and satisfies $\lim\limits_{\lambda \to -\infty} \Lambda(\lambda) = \infty$.
	\end{lem}
	
	\begin{proof}
		Let $\psi = \mathbf{1}_{[-1,0)} \psi_- + \mathbf{1}_{(0,1]} \psi_+$ be a solution to \eqref{Rayleigh function of CH_{m,0,l,0}}. Then $\psi_{\pm}$ satisfies
		\beq\left\{\bali\label{equation of psi_pm}
		- \theta \pa_y \left(\f{1}{\theta} \pa_y \psi_\pm\right) + \f{\theta}{u} \left(\f{u'}{\theta}\right)' \psi_\pm = \lambda \psi_\pm, \quad y \neq 0&,\\
		\psi_+'(0) - \psi_-'(0) = - \Lambda \psi(0) , \quad \psi_+(1) = \psi_-(-1) = 0 &.
		\eali\right.\eeq
		Let $\psi_\pm = c_\pm u \phi_1 T_\pm$, where $\phi_1$ is determined by \ref{solution of phi = u phi_1}. Then $T_\pm$ satisfies
		\beq\nonumber
		\f{d}{dy} \left(\f{u^2 \phi_1^2}{\theta} \f{d}{dy} T_\pm\right) = 0, \quad T_\pm (\pm 1) = 0.
		\eeq
		Equivalently,
		\beq\bali\nonumber
		T_\pm(y) 
		= \int_{\pm 1}^y \f{\theta}{u^2 \phi_1^2} \dz 
		= \int_{\pm 1}^y \f{\theta}{u^2}\left(\f{1}{\phi_1^2} -1\right) \dz + \int_{\pm 1}^y \f{\theta}{u^2} \dz.
		\eali\eeq
		Using the fact that $\phi_1(0) = 1$, we find
		\beq\bali\nonumber
		\psi_\pm(0) = \lim_{y \rightarrow 0} \psi_\pm(y) 
		= \lim_{y \rightarrow 0} c_\pm u(y) \int_{\pm 1}^y \f{\theta}{u^2} \dz = -c_\pm \f{\theta(0)}{u'(0)}.
		\eali\eeq
		For the continuity of $\psi$, we choose $c_+ = c_- = 1$. Then the $\psi$ can be expressed as
		\beq\nonumber
		\psi(y) =  \mathbf{1}_{[-1,0)} u(y) \phi_1(y) \int_{-1}^y \f{\theta}{u^2 \phi_1^2} \dz  + \mathbf{1}_{(0,1]} u(y) \phi_1(y) \int_{1}^y \f{\theta}{u^2 \phi_1^2} \dz.
		\eeq
		Furthermore, we have
		\beq\nonumber
		\psi_\pm'(0) = u'(0) \int_{\pm 1}^0 \f{\theta}{u^2} (\f{1}{\phi_1^2} -1) \dz + \f{d}{dy} \left(u(y) \int_{\pm 1}^y \f{\theta}{u^2} \right) \bigg|_{y=0},
		\eeq
		We need to verify that the second term in the above expression is finite. Observe that
		\beq\bali\nonumber
		u(y) \int_{\pm 1}^y \f{\theta}{u^2} 
		=& u(y) \int_{\pm 1}^y \f{\theta(z) - \theta(0)}{u^2(z)} \dz 
		+ u(y) \int_{\pm 1}^y \theta(0) \f{u'(0) - u'(z)}{u^2(z) u'(0)} \dz\\
		&+ u(y) \f{\theta(0)}{u'(0)} \int_{\pm 1}^y \f{u'(z)}{u^2(z)} \dz
		\eali\eeq
		We compute the integrals as follows:
		\beq\bali\nonumber
		\f{d}{dy} \left( u(y) \int_{\pm 1}^y \f{\theta(z) - \theta(0)}{u^2(z)} \dz  \right) \bigg|_{y=0} &= u'(0) \int_{\pm 1}^0 \f{\theta(z) - \theta(0)}{u^2(z)} \dz < \infty,\\
		\f{d}{dy} \left( u(y) \int_{\pm 1}^y \theta(0) \f{u'(0) - u'(z)}{u^2(z) u'(0)} \dz \right) \bigg|_{y=0} &= u'(0) \int_{\pm 1}^0 \theta(0) \f{u'(0) - u'(z)}{u^2(z) u'(0)} \dz < \infty,\\
		\f{d}{dy} \left( u(y) \f{\theta(0)}{u'(0)} \int_{\pm 1}^y \f{u'(z)}{u^2(z)} \dz \right) \bigg|_{y=0} &= \f{d}{dy} \left( \f{\theta(0)}{u'(0)} \f{u(y)}{u(\pm 1)} - \f{\theta(0)}{u'(0)} \right) \bigg|_{y=0}\\
		&= \f{\theta(0)}{u(\pm 1)} < \infty.
		\eali\eeq
		Let us define $b_\pm$ by
		\beq\nonumber
		b_\pm = \f{d}{dy} \left(u(y) \int_{\pm 1}^y \f{\theta}{u^2} \right) \bigg|_{y=0} < \infty,
		\eeq
		Then the second condition of \eqref{equation of psi_pm} becomes
		\beq\nonumber
		- \Lambda \psi(0) = { \psi_+'(0)} - \psi_-'(0) = -u'(0) \int_{-1}^{1} \f{\theta}{u^2} \left(\f{1}{\phi_1^2} -1\right) \dz + b_+ -b_-.
		\eeq
		Thus, $\Lambda$ is uniquely determined by
		\beq\nonumber
		\Lambda = \f{u'(0)}{\theta(0)}\left[ -u'(0) \int_{-1}^{1} \f{\theta}{u^2} \left(\f{1}{\phi_1^2} -1\right) \dz + b_+ -b_- \right].
		\eeq
		
		Finally, we show the monotonicity of $\Lambda(\lambda)$. Differentiating with respect to $\lambda$, we obtain
		\beq\nonumber
		\pa_\lambda \Lambda = \f{2u'(0)^2}{\theta(0)} \int_{-1}^1 \f{\theta(z)}{u^2(z) \phi_1^2(z)} \f{\pa_\lambda \phi_1(z)}{\phi_1(z)} \dz.
		\eeq
		Recall that $\phi_1(y)$ satisfies
		\beq\label{Lemma: Lambda(lambda): eq1}
		\f{d}{dy} \left(\f{u^2}{\theta} \f{d}{dy} \phi_1\right) = - \lambda \f{u^2}{\theta} \phi_1.
		\eeq
		Noting that $\phi_1(y) \geq 1$, we may set $\pa_\lambda \phi_1 = K \phi_1$. Differentiating \eqref{Lemma: Lambda(lambda): eq1} with respect to $\lambda$, we find that $K$ satisfies
		\beq\nonumber
		\f{d}{dy} \left(\f{u^2}{\theta} \phi_1^2 \f{d}{dy} K\right) = - \f{u^2}{\theta} \phi_1^2, \quad K(0)=0.
		\eeq
		Hence, $K$ can be expressed as
		\beq\nonumber
		K(y) = -\int_0^y \f{\theta(y')}{u^2(y') \phi_1^2(y')} \int_0^{y'} \f{u^2(y'') {  \phi_1^2(y'')}} {\theta(y'')} \dy'' \dy'<0.
		\eeq
		Consequently, we obtain
		\beq\nonumber
		\pa_\lambda \Lambda 
		= \f{2u'(0)^2}{\theta(0)} \int_{-1}^1 \f{\theta(z)}{u^2(z) \phi_1^2(z)} K(z) \dz <0.
		\eeq
		From \eqref{Rayleigh function of CH_{m,0,l,0}}, $\lambda$ admits the variational characterization
		\beq\nonumber
		\lambda = \min_{\substack{\int_{-1}^1 \f{1}{\theta}|\psi|^2dy = 1, \\ \psi \in H_0^1(-1,1)}} \int_{-1}^1 \f{1}{\theta} |\pa_y \psi|^2 + \f{1}{u}\left(\f{u'}{\theta}\right)'|\psi|^2 \dy - \f{\Lambda}{\theta(0)} |\psi(0)|^2.
		\eeq
		We find that $\lambda \rightarrow -\infty$ when $\Lambda \rightarrow \infty$. 
	\end{proof}
	
	\subsection{The convergence of the perturbed distorted Rayleigh operator}\label{sec:The convergence of the perturbed}
	In this subsection, we investigate the convergence of the perturbed distorted Rayleigh operator $\CH_{m, \gamma, l, \eta}$ to its limit $\CH_{m, l}$ for sufficiently small $\gamma$ and $\eta$. More precisely, using the energy method, we show that the difference between the minimal eigenvalues of $\CH_{m, \gamma, l, \eta}$ and $\CH_{m, l}$ is controlled by $\gamma$ and $\eta$.
	
	For every $\psi \in H_0^1(-1,1)$, we define the associated energy functionals of $\CH_{m, l}$ and $\CH_{m, \gamma, l, \eta}$ as follows:
	\beq\bali\nonumber
	E_{m,l}[\psi] &= \int_{-1}^1 \f{1}{\theta} |\f{d}{dy} \psi|^2 + \f{1}{u}\left(\f{u'}{\theta}\right)'|\psi|^2 \dy - \f{\Lambda}{\theta(0)} |\psi(0)|^2,\\
	E_{m,\gamma,l,\eta}[\psi] &=  \int_{-1}^1 \f{1}{\theta_{l,\eta}} |\f{d}{dy}  \psi|^2 
	+ \f{1}{u_{m,\gamma}} \left(\f{u_{m,\gamma}'}{\theta_{l,\eta}} \right)' |\psi|^2 \dy.
	\eali\eeq
	Let $\lambda_{m,l}$ and $\lambda_{m,\gamma,l,\eta}$ denote the smallest eigenvalues of $\CH_{m,l}$ and $\CH_{m,\gamma,l,\eta}$ given by
	\beq\bali\label{expression of lambda_{m,l}}
	\lambda_{m,l} &= \min_{\psi \in \Omega_\theta} E_{m,l}[\psi], && \Omega_\theta = \{\psi \in H_0^1(-1,1): \int_{-1}^1 \f{1}{\theta}|\psi|^2 \dy = 1\},\\
	\lambda_{m,\gamma,l,\eta} 
	&= \min_{\psi \in \Omega_{l, \eta}}  E_{m,\gamma,l,\eta}[\psi], && \Omega_{l, \eta} = \{\psi \in H_0^1(-1,1): \int_{-1}^1 \f{1}{\theta_{l,\eta}}|\psi|^2 \dy = 1\}.
	\eali\eeq
	With the help of Proposition \ref{solution of phi = u phi_1}, we first establish the uniform boundedness of the minimizer for the energy $E_{m, \gamma, l, \eta}$.
	
	\begin{lem}\label{lemma: boundedness of psi_{gamma,mu}}
		Assume that $|m|,|l| \leq M$ and that $\gamma, \eta \ll 1$ are sufficiently small. Let $\lambda_{m,\gamma,l,\eta} < 0$ and let $\psi_{\gamma,\eta} \in \Omega_{l, \eta}$ be a minimizer such that
		\beq\nonumber
		\lambda_{m,\gamma,l,\eta} = E_{m,\gamma,l,\eta}[\psi_{\gamma,\eta}].
		\eeq 
		Then there exists a constant $C>0$ such that
		\beq\nonumber
		|\psi_{\gamma,\eta}(0)| \geq \f{1}{C}, \quad \|\psi_{\gamma,\eta}\|_{W^{1, \infty}} \leq C. 
		\eeq
	\end{lem}
	
	\begin{proof}
		We first show that $\lambda_{m,\gamma,l,\eta}$ is bounded from below for bounded $m,l$ and small $\gamma,\eta$. Applying the Gagliardo-Nirenberg inequality yields
		\beq\bali\nonumber
		\int_{-1}^1  \f{1}{u_{m,\gamma}} \left( \f{u_{m,\gamma}'}{\theta_{l,\eta}} \right)' |\psi|^2 \dy \leq& \left\|\f{1}{u_{m,\gamma}} \left( \f{u_{m,\gamma}'}{\theta_{l,\eta}} \right)' \right\|_{L^1} \|\psi\|_{L^\infty}^2 \leq C \|\psi\|_{L^2} \|\f{d}{dy} \psi\|_{L^2}\\
		\leq& \f{1}{2} \int_{-1}^1 \f{1}{\theta_{l,\eta}} |\f{d}{dy}  \psi|^2 \dy + C \int_{-1}^1 \f{1}{\theta_{l,\eta}}|\psi|^2 \dy.
		\eali\eeq
		So $\lambda_{m,\gamma,l,\eta}$ has the lower bound
		\beq\bali\nonumber
		\lambda_{m,\gamma,l,\eta} 
		= \min_{\psi \in \Omega_{  l,\eta}}  E_{m,\gamma,l,\eta}[\psi] \geq \f{1}{2} \int_{-1}^1 \f{1}{\theta_{l,\eta}} |\f{d}{dy}  \psi|^2 \dy - C \geq -C.
		\eali\eeq
		From Lemma \ref{Lemma: Lambda(lambda)},  $\psi_{\gamma,\eta}$ can be expressed as
		\beq\bali\nonumber
		{  c_{\gamma,\eta}} \psi_{\gamma,\eta}(y) 
		= \mathbf{1}_{[-1,0)} u_{m,\gamma}(y) \phi(y) \int_{-1}^y \f{\theta_{l,\eta}}{u_{m,\gamma}^2 \phi^2} \dz  
		+ \mathbf{1}_{(0,1]} u_{m,\gamma}(y) \phi(y) \int_{1}^y \f{\theta_{l,\eta}}{u_{m,\gamma}^2 \phi^2} \dz,
		\eali\eeq
		{  where $c_{\gamma,\eta}>0$ is determined by the weighted normalization $\int_{-1}^1 \f{|\psi_{\gamma,\eta}(y)|^2}{\theta_{l,\eta}(y)}\,\dy=1$, and there is}
		\beq\nonumber
		{  c_{\gamma,\eta}} \psi_{\gamma,\eta}(0) = - \f{\theta_{l,\eta}(0)}{u_{m,\gamma}'(0)}.
		\eeq
		Here, $\phi$ is the solution to the integral equation
		\beq\nonumber
		\phi(y) = 1 - \lambda_{m,\gamma,l,\eta}  \int_{0}^y \f{\theta_{l,\eta}(z')}{u_{m,\gamma}(z')^2} \int_{0}^{z'} \f{u_{m,\gamma}(z'')^2}{\theta_{l,\eta}(z'')} \phi(z'') \dz''\dz'.
		\eeq
		By Proposition \ref{solution of phi = u phi_1}, there exists $C>0$ such that
		\beq\nonumber
		1 \leq \phi(y) = e^{{  \int_{\min(0,y)}^{\max(0,y)}} |\f{\pa_y \phi}{\phi}| \dz} \leq e^{C|y|},
		\eeq
		and
		\beq\nonumber
		\left| \f{d}{dy} \phi(y) \right| \leq C \sqrt{- \lambda_{m,\gamma,l,\eta}}\phi(y) \leq C e^{C|y|}.
		\eeq
		We expand $\psi_{\gamma,\eta}(y)$ as follows:
		\beq\bali\nonumber
		&\mathbf{1}_{\BR^\pm \cap [-1,1]} {  c_{\gamma,\eta}} \psi_{\gamma,\eta}(y) \\
		& \qquad = u_{m,\gamma}(y) \phi(y) \int_{\pm 1}^y \f{\theta_{l,\eta}}{u_{m,\gamma}^2} \left(\f{1}{\phi^2} - 1\right) \dz 
		+ u_{m,\gamma}(y) \phi(y) \int_{\pm 1}^y \f{\theta_{l,\eta}(z) - \theta_{l,\eta}(0)}{u_{m,\gamma}^2(z)} \dz \\
		& \qquad \quad + u_{m,\gamma}(y) \phi(y) \f{\theta_{l,\eta}(0)}{u_{m,\gamma}'(0)} \int_{\pm 1}^y \f{u_{m,\gamma}'(0) - u_{m,\gamma}'(z)}{u_{m,\gamma}^2(z)} \dz\\
		& \qquad \quad + \phi(y) \left(\f{\theta_{l,\eta}(0)}{u_{m,\gamma}'(0)} \f{u_{m,\gamma}(y)}{u_{m,\gamma}(\pm 1)} - \f{\theta_{l,\eta}(0)}{u_{m,\gamma}'(0)} \right).
		\eali\eeq
		Using the estimates
		\beq\nonumber
		\| u_{m,\gamma} - u\|_{W^{1, \infty}} \leq C |m \gamma|, \quad \| \theta_{l,\eta} - \theta\|_{L^\infty} \leq C |l \eta|,
		\eeq
		a direct calculation shows that
		\beq\nonumber
		\|\psi_{\gamma,\eta}\|_{W^{1, \infty}} \leq C,
		\eeq
		which completes the proof.
	\end{proof}

	Next, we derive a uniform estimate on the difference between $\lambda_{m,l}$ and $\lambda_{m,\gamma,l,\eta}$.

	\begin{lem}\label{lemma for smallness of gamma and eta}
		For every $M > 0$, there exists $\delta > 0$ such that if $|m|, |l| \leq M$ and $\gamma, \eta \leq \delta$, then it holds
		\beq\nonumber
		|\lambda_{m,l} - \lambda_{m,\gamma,l,\eta}| \leq C |m \gamma| + C|l \eta|,
		\eeq
		where the constant $C$ is independent of $M$.
	\end{lem}
	
	\begin{proof}
		Let $\psi_{\gamma,\eta} \in \Omega_{l, \eta}$ be a minimizer such that
		\beq\nonumber
		\lambda_{m,\gamma,l,\eta} = E_{m,\gamma,l,\eta}[\psi_{\gamma,\eta}].
		\eeq 
		Set $\tilde{\psi}_{\gamma,\mu} = A^{-1} \psi_{\gamma,\eta} $ such that $\tilde{\psi}_{\gamma,\mu} \in \Omega_\theta$. For the constant $A$, we have the estimate
		\beq\bali\nonumber
		A^2 - 1 =& \int_{-1}^1 \left(\f{1}{\theta} - \f{1}{\theta_{l,\eta}}\right) |\psi_{\gamma,\eta}|^2 \dy 
		= \int_{-1}^1 \f{l \eta {  u'(y)}}{\theta(y) \theta_{l,\eta}(y)} \digamma\left(\f{y}{\eta}\right) |\psi_{\gamma,\eta}(y)|^2 \dy\\
		\leq& \left\|\f{l \eta {  u'(y)}}{\theta(y)} \digamma\left(\f{y}{\eta}\right)\right\|_{L^\infty(-1,1)} \int_{-1}^1 \f{1}{\theta_{l,\eta}}|\psi_{\gamma,\eta}|^2 \dy \leq C|l \eta|.
		\eali\eeq
		We proceed as follows:
		\beq\bali\nonumber
		A^2 \lambda_{m,l} \leq& A^2 E_{m,l}[\tilde{\psi}_{\gamma,\mu}] = \lambda_{m,\gamma,l,\eta} + E_{m,l}[\psi_{\gamma,\eta}] - E_{m,\gamma,l,\eta}[\psi_{\gamma,\eta}]\\
		=& \lambda_{m,\gamma,l,\eta} 
		+ \int_{-1}^1 \left(\f{1}{\theta} - \f{1}{\theta_{l,\eta}} \right) |\pa_y \psi_{\gamma,\eta}|^2 \dy \\
		&+ \int_{-1}^1 \left[ \f{1}{u}\left(\f{u'}{\theta}\right)' - \f{1}{u_{m,\gamma}} \left(\f{u_{m,\gamma}'}{\theta_{l,\eta}} \right)'\right] |\psi_{\gamma,\eta}|^2 \dy
		- \f{\Lambda}{\theta(0)}\psi_{\gamma,\eta}(0)^2\\
		=& \lambda_{m,\gamma,l,\eta} 
		+ \int_{-1}^1 \left(\f{1}{\theta} - \f{1}{\theta_{l,\eta}} \right) |\pa_y \psi_{\gamma,\eta}|^2 \dy\\
		&+ \int_{-1}^1 \left(\f{u_{m,\gamma}''\theta_{l,\eta} - u_{m,\gamma}'\theta_{l,\eta}'}{u \theta^2} - \f{u_{m,\gamma}''\theta_{l,\eta} - u_{m,\gamma}'\theta_{l,\eta}'}{u_{m,\gamma} \theta_{l,\eta}^2}\right)|\psi_{\gamma,\eta}|^2 \dy\\
		&+ \int_{-1}^1\left[ \left(\f{u''\theta - u'\theta'}{u \theta^2} - \f{u_{m,\gamma}''\theta_{l,\eta} - u_{m,\gamma}'\theta_{l,\eta}'}{u \theta^2} \right)|\psi_{\gamma,\eta}|^2 - \f{\Lambda}{\theta(0)} \psi_{\gamma,\eta}(0)^2 \delta(0) \right] \dy\\
		:=& \lambda_{m,\gamma,l,\eta} + I + J + K.
		\eali\eeq
		We estimate the terms $I$, $J$, and $K$ separately.
		For the term $I$, applying the Gagliardo-Nirenberg inequality yields
		\beq\bali\nonumber
		\int_{-1}^1 \f{1}{\theta_{l,\eta}} |\pa_y \psi_{\gamma,\eta}|^2 \dy =& \lambda_{m,\gamma,l,\eta} - \int_{-1}^1 \f{1}{u_{m,\gamma}} \left(\f{u_{m,\gamma}'}{\theta_{l,\eta}} \right)' |\psi_{\gamma,\eta}|^2 \dy\\
		\leq& \lambda_{m,\gamma,l,\eta} + C \left\|\f{1}{u_{m,\gamma}} \left(\f{u_{m,\gamma}'}{\theta_{l,\eta}} \right)'\right\|_{L^1} \|\psi_{\gamma,\eta}\|_{L^\infty}^2\\
		\leq& \lambda_{m,\gamma,l,\eta} + C \|\psi_{\gamma,\eta}\|_{L^2} \|\pa_y \psi_{\gamma,\eta}\|_{L^2}\\
		\leq& \lambda_{m,\gamma,l,\eta} + C \left( \int_{-1}^1 \f{1}{\theta_{l,\eta}} |\pa_y \psi_{\gamma,\eta}|^2 \dy \right)^\f{1}{2}.
		\eali\eeq
		Given that $\gamma, \eta$ are small and $m, l$ are bounded, $\lambda_{m,\gamma,l,\eta}$ is bounded from below. It follows that
		\beq\nonumber
		\int_{-1}^1 \f{1}{\theta_{l,\eta}} |\pa_y \psi_{\gamma,\eta}|^2 \dy \leq C.
		\eeq
		Thus,
		\beq\nonumber
		I = l \eta \int_{-1}^1 {  \f{u'(y)\digamma\left(\f{y}{\eta}\right)}{\theta(y)\theta_{l,\eta}(y)} }  |\pa_y \psi_{\gamma,\eta}(y)|^2 \dy \leq C |l \eta|.
		\eeq
		For the term $J$, noting that
		\beq\bali\nonumber
		{ \big|} \f{1}{u \theta^2} - \f{1}{u_{m,\gamma} \theta_{l,\eta}^2} { \big|} = { \big|} \f{1}{u \theta^2} \f{u_{m,\gamma} \theta_{l,\eta}^2 - u \theta^2}{u_{m,\gamma} \theta_{l,\eta}^2} { \big|} \leq \f{C}{{  |u| }\theta^2}\left(|m \gamma| + |l \eta|\right),
		\eali\eeq
		we have
		\beq\bali\nonumber
		J \leq& C \left(|m \gamma| + |l \eta|\right) \int_{-1}^1 \left|\f{u_{m,\gamma}''\theta_{l,\eta} - u_{m,\gamma}'\theta_{l,\eta}'}{u \theta^2} \right| |\psi_{\gamma,\eta}|^2 \dy\\
		\leq& C \left(|m \gamma| + |l \eta|\right) \left\|\f{u_{m,\gamma}''\theta_{l,\eta} - u_{m,\gamma}'\theta_{l,\eta}'}{u \theta^2}\right\|_{L^1} \|\psi_{\gamma,\eta}\|_{L^2} \|\pa_y \psi_{\gamma,\eta}\|_{L^2}\\
		\leq& C \left(|m \gamma| + |l \eta|\right).
		\eali\eeq
		For the term $K$, using condition \eqref{eq: Gamma and digamma's condition}, we deduce that
		\beq\bali\nonumber
		\f{\Lambda}{\theta(0)}\psi_{\gamma,\eta}(0)^2 =& \int_{\BR} \left(-\f{m}{u'(0)y} \Gamma'(y) + \f{l u'(0)}{\theta(0)^2 y} \digamma'(y)\right) \psi_{\gamma,\eta}(0)^2 \dy\\
		=& -\int_{-1}^1 \f{m}{u'(0)y} \Gamma'\left(\f{y}{\gamma}\right) \psi_{\gamma,\eta}(0)^2 \dy + \int_{-1}^1 \f{l u'(0)}{\theta(0)^2 y} \digamma'\left(\f{y}{{  \eta}}\right) \psi_{\gamma,\eta}(0)^2 \dy\\
		&+ \left( -\f{m}{u'(0)}\int_{\BR \backslash [-\f{1}{\gamma}, \f{1}{\gamma}]} \f{\Gamma'(y)}{y} \dy + \f{l u'(0)}{\theta(0)^2} \int_{\BR \backslash [-\f{1}{\eta}, \f{1}{\eta}]} \f{\digamma'(y)}{y} \dy\right) \psi_{\gamma,\eta}(0)^2.
		\eali\eeq
		We decompose $K$ as
		\beq\bali\nonumber
		K =& -\int_{-1}^1 \f{1}{u \theta^2}\left[u'' \tilde{\digamma} + \tilde{\Gamma}'' \tilde{\digamma} - \tilde{\Gamma}' \theta' - \tilde{\Gamma}' \tilde{\digamma}' + m \gamma \theta \theta' \Gamma\left(\f{y}{\gamma}\right) - l \eta u' u'' \digamma\left(\f{y}{\eta}\right) \right] |\psi_{\gamma,\eta}|^2 \dy \\
		& - m \int_{-1}^1 \f{\Gamma'\left(\f{y}{\gamma}\right)}{y} \left[\f{y}{u(y)} \psi_{\gamma,\eta}(y)^2 - \f{1}{u'(0)} \psi_{\gamma,\eta}(0)^2 \right] \dy \\
		& + l \int_{-1}^1 \f{\digamma'\left(\f{y}{{  \eta}}\right)}{y} \left[\f{u'(y)^2 y}{u(y) \theta(y)^2} \psi_{\gamma,\eta}(y)^2 - \f{u'(0)}{\theta(0)^2} \psi_{\gamma,\eta}(0)^2 \right] \dy\\
		&- \left[ -\f{m}{u'(0)}\int_{\BR \backslash [-\f{1}{\gamma}, \f{1}{\gamma}]} \f{\Gamma'(y)}{y} \dy + \f{l u'(0)}{\theta(0)^2} \int_{\BR \backslash [-\f{1}{\eta}, \f{1}{\eta}]} \f{\digamma'(y)}{y} \dy\right] \psi_{\gamma,\eta}(0)^2\\
		:=& K_1 + K_2 + K_3 + K_4.
		\eali\eeq
		The term $K_1$ can be estimated by direct calculation:
		\beq\bali\nonumber
		K_1 \leq& \left\|\f{1}{u \theta^2}\left[u'' \tilde{\digamma} + \tilde{\Gamma}'' \tilde{\digamma} - \tilde{\Gamma}' \theta' - \tilde{\Gamma}' \tilde{\digamma}' + m \gamma \theta \theta' \Gamma\left(\f{y}{\gamma}\right) - l \eta u' u'' \digamma\left(\f{y}{\eta}\right) \right]\right\|_{L^1} \|\psi_{\gamma,\eta}\|_{L^\infty}^2\\
		\leq&  C \left(|m \gamma| + |l \eta|\right),
		\eali\eeq
		while the $K_2$ and $K_3$ are bounded by
		\beq\bali\nonumber
		K_2 \leq& |m| \left\|\f{y}{u(y)} \psi_{\gamma,\eta}(y)^2\right\|_{W^{1,\infty}} \int_{-1}^1 \left|\Gamma'\left(\f{y}{\gamma}\right)\right| \dy \leq C |m \gamma|,\\
		K_3 \leq& |l| \left\|\f{u'(y)^2 y}{u(y) \theta(y)^2} \psi_{\gamma,\eta}(y)^2\right\|_{W^{1,\infty}} \int_{-1}^1 \left|\digamma'\left(\f{y}{\eta}\right)\right| \dy \leq C |l \eta|.
		\eali\eeq
		From the conditions on $\Gamma$ and $\digamma$, the term $K_4$ is controlled by
		\beq\bali\nonumber
		&\left| -\f{m}{u'(0)}\int_{\BR \backslash [-\f{1}{\gamma}, \f{1}{\gamma}]} \f{\Gamma'(y)}{y} \dy + \f{l u'(0)}{\theta(0)^2} \int_{\BR \backslash [-\f{1}{\eta}, \f{1}{\eta}]} \f{\digamma'(y)}{y} \dy\right| |\psi_{\gamma,\eta}(0)|^2\\
		& \qquad \leq C \left(|m \gamma| + |l \eta|\right) \|\psi_{\gamma,\eta}\|_{L^\infty}^2 \leq  C \left(|m \gamma| + |l \eta|\right).
		\eali\eeq
		Combining these estimates, we obtain $K\leq C \left(|m \gamma| + |l \eta|\right)$. Thus,
		\beq\bali\nonumber
		\lambda_{m,l} - \lambda_{m,\gamma,l,\eta} 
		\leq& (1 - A^2)\lambda_{m,l} + I + J + K \leq C \left(|m \gamma| + |l \eta|\right).
		\eali\eeq
		Similarly, we can show that
		\beq\bali\nonumber
		\lambda_{m,\gamma,l,\eta} - \lambda_{m,l} \leq C \left(|m \gamma| + |l \eta|\right).
		\eali\eeq
		This concludes the proof.
	\end{proof}

	\begin{lem}
		For any $M > 0$, there exist constants $C, c, \delta > 0$ such that if $|m|, |\tilde{m}|, |l|, |\tilde{l}| \leq M$, $\gamma, \eta \leq \delta$ and $\Lambda_{m - \tilde{m}, l - \tilde{l}} \geq 0$ , then
		\beq\nonumber
		\lambda_{m,\gamma,l,\eta} - \lambda_{\tilde{m},\gamma,\tilde{l},\eta} 
		\leq - c \Lambda_{m - \tilde{m}, l - \tilde{l}} + {  C(\gamma+\eta)\left( |m - \tilde{m}| + |l - \tilde{l}| \right)}.
		\eeq
		Consequently, for any $m_1 < m_2$ and $l_1 < l_2$, we have
		\beq\label{monotonicity of lambda with respect to m and l}
		\lambda_{m_2,\gamma,l,\eta} < \lambda_{m_1,\gamma,l,\eta}, \quad \lambda_{m,\gamma,l_2,\eta} < \lambda_{m,\gamma,l_1,\eta}
		\eeq
	\end{lem}
	
	\begin{proof}
		Recall $\Lambda_{m - \tilde{m}, l - \tilde{l}}$ is defined in \eqref{definition of Lambda_m,l}. Using the definition \eqref{expression of lambda_{m,l}} of $\lambda_{m,\gamma,l,\eta}$, we estimate the difference between the eigenvalues by comparing their corresponding energies. For convenience, we write
		\beq\nonumber
		u_{m,\gamma} = u + m \bar{\Gamma}_\gamma, \quad \theta_{l,\eta} = \theta + l \bar{\digamma}_\eta.
		\eeq
		Let $\psi_{\tilde{m},\tilde{l}} \in \Omega_{\tilde{l}, \eta}$ be a minimizer of $E_{\tilde{m},\gamma,\tilde{l},\eta}$. Define the scaled function $\tilde{\psi}_{m,\tilde{l}} = A^{-1} \psi_{\tilde{m},\tilde{l}}$ such that $\tilde{\psi}_{m,\tilde{l}}\in \Omega_{l, \eta}$. We have the estimate 
		$$|A^2 - 1| \leq C \eta |l - \tilde{l}|.$$
		Then,
		\beq\bali\nonumber
		& A^2 \lambda_{m, \gamma, l, \eta} - \lambda_{\tilde{m}, \gamma, \tilde{l}, \eta} \leq E_{m, \gamma, l, \eta} [\psi_{\tilde{m}, \tilde{l}}] - E_{\tilde{m}, \gamma, \tilde{l}, \eta} [\psi_{\tilde{m}, \tilde{l}}]\\
		=& \int_{-1}^1 \left( \f{1}{\theta_{l,\eta}} - \f{1}{\theta_{\tilde{l},\eta}}\right) |\pa_y \psi_{\tilde{m}, \tilde{l}}|^2 \dy 
		+ \int_{-1}^1 \left[ \f{1}{u_{m, \gamma}} \left( \f{u_{m, \gamma}'}{\theta_{l,\eta}} \right)' - \f{1}{u_{\tilde{m},\gamma}} \left( \f{ u_{\tilde{m}, \gamma}' }{ \theta_{l, \eta} } \right)' \right] |\psi_{\tilde{m},\tilde{l}}|^2 \dy\\
		&+ \int_{-1}^1 \left[ \f{1}{u_{\tilde{m},\gamma}} \left( \f{ u_{\tilde{m}, \gamma}' }{ \theta_{l, \eta} } \right)' - \f{1}{u_{\tilde{m},\gamma}} \left(\f{ u_{\tilde{m}, \gamma}' }{ \theta_{\tilde{l}, \eta} } \right)' \right] |\psi_{\tilde{m},\tilde{l}}|^2 \dy \\
		:=& I + J + K.
		\eali\eeq
		The first term is directly bounded by
		\beq\bali\label{equation of difference term I}
		I = (\tilde{l} - l) \int_{-1}^1 \f{\tilde{\digamma}_\eta}{\theta_{l,\eta} \theta_{\tilde{l},\eta}}  |\pa_y \psi_{\tilde{m},\tilde{l}}|^2 \dy \leq C\eta |l - \tilde{l}|.
		\eali\eeq
		For the second term, we compute the derivative with respect to $m$ of the integrand:
		\beq\bali\nonumber
		\pa_m \left[ \f{1}{u_{m,\gamma}} \left(\f{u_{m,\gamma}'}{\theta_{l,\eta}} \right)' \right] 
		= \f{1}{u_{m,\gamma}^2 \theta_{l,\eta}^2} \left[u \theta_{l,\eta} \bar{\Gamma}_\gamma'' - u'' \theta_{l, \eta} \bar{\Gamma}_\gamma - \left(\bar{\Gamma}_\gamma' u - \bar{\Gamma}_\gamma u'\right) \theta_{l, \eta}' \right].
		\eali\eeq 
		Using the estimates established in Lemma \ref{lemma for smallness of gamma and eta}, we obtain
		\beq\bali\nonumber
		\int_{-1}^1 \f{u \theta_{l,\eta} \bar{\Gamma}_\gamma''}{u_{m,\gamma}^2 \theta_{l,\eta}^2} |\psi|^2 \dy + \f{1}{u'(0)} |\psi(0)|^2 \leq C {  (\gamma+\eta)},
		\eali\eeq
		and
		\beq\bali\nonumber
		\int_{-1}^1 \f{u'' \theta_{l, \eta} \bar{\Gamma}_\gamma}{u_{m,\gamma}^2 \theta_{l,\eta}^2} |\psi|^2 \dy 
		\leq& \left\|\f{u'' \theta_{l, \eta} \bar{\Gamma}_\gamma}{u_{m,\gamma}^2 \theta_{l,\eta}^2} \right\|_{L^\infty} \|\psi\|_{L^2}^2 
		\leq C \gamma,\\
		\int_{-1}^1 \f{\left(\bar{\Gamma}_\gamma' u - \bar{\Gamma}_\gamma u'\right) \theta_{l, \eta}'}{u_{m,\gamma}^2 \theta_{l,\eta}^2} |\psi|^2 \dy
		\leq& \left\| \f{\bar{\Gamma}_\gamma' u - \bar{\Gamma}_\gamma u'}{u_{m,\gamma} \theta_{l,\eta}^2} \right\|_{L^\infty} \left\|\f{\theta_{l, \eta}'}{u_{m,\gamma}}\right\|_{L^1} \|\psi\|_{L^\infty}^2 \leq C\gamma.
		\eali\eeq
		Thus,
		\beq\bali\label{equation of difference of lambda up to m}
		J \leq& -\f{m - \tilde{m}}{u'(0)} |\psi_{\tilde{m},l}(0)|^2 + {  C(\gamma+\eta) |m - \tilde{m}|}.
		\eali\eeq
		We now consider the third term. We compute the derivative with respect to $l$ as follows:
		\beq\bali\nonumber
		\pa_l \left[ \f{1}{u_{m,\gamma}} \left(\f{u_{m,\gamma}'}{\theta_{l,\eta}} \right)' \right] 
		=& \f{\left(l \tilde{\digamma}_\eta - \theta\right) u_{m,\gamma}' \tilde{\digamma}_\eta'}{u_{m,\gamma} \theta_{l,\eta}^3} 
		+ \f{2 u_{m,\gamma}' \theta' \tilde{\digamma}_\eta }{u_{m,\gamma} \theta_{l,\eta}^3}  - \f{\left(l \tilde{\digamma}_\eta^2 + \theta \tilde{\digamma}_\eta \right) u_{m,\gamma}''}{u_{m,\gamma} \theta_{l,\eta}^3}.
		\eali\eeq 
		Using the estimates from Lemma \ref{lemma for smallness of gamma and eta}, we find
		\beq\bali\nonumber
		\int_{-1}^1 \f{\left(l \tilde{\digamma}_\eta - \theta\right) u_{m,\gamma}' \tilde{\digamma}_\eta'}{u_{m,\gamma} \theta_{l,\eta}^3} |\psi|^2 \dy 
		+ \f{u'(0)}{\theta(0)^2} |\psi(0)|^2 \leq C\eta,
		\eali\eeq
		and
		\beq\bali\nonumber
		\int_{-1}^1 \f{2 u_{m,\gamma}' \theta' \tilde{\digamma}_\eta }{u_{m,\gamma} \theta_{l,\eta}^3}  |\psi|^2 \dy 
		\leq& \left\|\f{2 u_{m,\gamma}' \theta' \tilde{\digamma}_\eta }{u_{m,\gamma} \theta_{l,\eta}^3} \right\|_{L^\infty} \|\psi\|_{L^2}^2 \leq C \eta,\\
		\int_{-1}^1 \f{\left(l \tilde{\digamma}_\eta^2 + \theta \tilde{\digamma}_\eta \right) u_{m,\gamma}''}{u_{m,\gamma} \theta_{l,\eta}^3} |\psi|^2 \dy
		\leq& \left\|\f{u_{m,\gamma}''}{u_{m,\gamma}} \right\|_{L^1} \left\|\f{l \tilde{\digamma}_\eta^2 + \theta \tilde{\digamma}_\eta }{ \theta_{l,\eta}^3} \right\|_{L^\infty} \|\psi\|_{L^\infty}^2 \leq C \eta.
		\eali\eeq
		Thus,
		\beq\label{equation of difference of lambda up to l}
		K \leq - \f{u'(0)}{\theta(0)^2} (l - \tilde{l}) | \psi_{\tilde{m},\tilde{l}}(0)|^2 + C\eta |l - \tilde{l}|.
		\eeq
		Combining \eqref{equation of difference term I}-\eqref{equation of difference of lambda up to l} and Proposition \ref{solution of phi = u phi_1}, we obtain
		\beq\bali\nonumber
		\lambda_{m,\gamma,l,\eta} - \lambda_{\tilde{m},\gamma,\tilde{l},\eta} \leq&  -\f{|\psi_{\tilde{m},l}(0)|^2}{\theta(0)} \Lambda_{m - \tilde{m}, l - \tilde{l}}  + C \eta |l - \tilde{l}| + {  C(\gamma+\eta) |m - \tilde{m}|} \\
		\leq& - c \Lambda_{m - \tilde{m}, l - \tilde{l}} + {  C(\gamma+\eta)\left( |m - \tilde{m}| + |l - \tilde{l}| \right)}.
		\eali\eeq
		For sufficiently small $\eta$ and $\gamma$, the inequality \eqref{monotonicity of lambda with respect to m and l} follows.
	\end{proof}

	\subsection{The Cat's-eyes structure near the modified shear flow}
	Based on the spectral properties of the perturbed, distorted Rayleigh operator, we complete the proof of the first theorem in this subsection. To ensure the Long equation is well-defined, we first construct modified profiles for $(u, \theta)$. Let $\chi_\delta \in C_c^\infty(\mathbb{R})$ be a smooth cut-off function satisfying
	$\chi_\delta \equiv 1$ on $(-\delta, \delta)$ and ${\rm supp}\, \chi_\delta \subset (-2\delta, 2\delta)$.
	Recall that the modified velocity $u^\delta$ and density $\theta^\delta$ defined by
	\beq
    \begin{aligned}
	u^\delta =& \chi_\delta(y) u'(0)y + \left(1 - \chi_\delta(y)\right) u(y),\\
	\theta^\delta =& \chi_\delta(y) \theta(0) + \left(1 - \chi_\delta(y)\right) \theta(y),
    \end{aligned}
	\eeq
	and the perturbed profiles are considered as follows:
	\beq\bali\nonumber
	u^\delta_{m, \gamma}(y) =& u^\delta(y) + m \gamma \int_0^y \theta^\delta(y') \Gamma(\f{y'}{\gamma}) \dy' ,\\
	\theta^\delta_{l, \eta}(y) =& \theta^\delta(y) + l \eta (u^\delta)'(y) \digamma(\f{y}{\eta}) ,
	\eali\eeq
	where $\Gamma$ and $\digamma$ satisfies condition \eqref{eq: Gamma and digamma's condition}. The distance of $(u^\delta_{m, \gamma}, \theta^\delta_{l, \eta})$ to the original profiles $(u, \theta)$ is quantified in the following lemma.
	
	\begin{lem}\label{lemma of distance between u,theta and u^delta, theta^delta}
		For $0 < \delta \ll 1$ and $k \in \BN$, there exists $C>0$ such that
		\beq\bali\label{equation of distance between u and u^delta}
		\| u^\delta - u\|_{H^{\f{5}{2} - s}} +\| \theta^\delta - \theta\|_{H^{\f{3}{2} - s}} \leq C \delta^s.
		\eali\eeq
		Furthermore, for the perturbed profiles, we have
		\beq\bali\nonumber
		\| u^\delta_{m, \gamma} - u^\delta \|_{H^{\f{5}{2} - s}} +\| \theta^\delta_{l, \eta} - \theta^\delta\|_{H^{\f{3}{2} - s}} \leq C \left(\delta^s + |m| \gamma^s + |l| \eta^s \right) .
		\eali\eeq
	\end{lem}
	
	\begin{proof}
		By the definition of $\chi_\delta$, we have $\| \chi_\delta\|_{C^k(-1,1)} \leq C \delta^{-k}$. We deduce  that 
		\beq\bali\nonumber
		\| u - u^\delta\|_{C^k} &= \| \chi_\delta(y) (u(y) - u'(0)y)\|_{C^k} \leq C \delta^{2-k}.
		\eali\eeq
		Taking $k = 2,3$, we obtain
		\beq\bali\nonumber 
		\|u - u^\delta\|_{H^2} \leq C \delta^\f{1}{2},\ \quad  \| u - u^\delta\|_{H^3} \leq C \delta^{-\f{1}{2}},.
		\eali\eeq
		Then \eqref{equation of distance between u and u^delta} follows from the interpolation inequality:
		\beq\bali\nonumber 
		\|u - u^\delta\|_{H^{\f{5}{2}-\tau}} \leq C \|u - u^\delta\|_{H^2}^{\f{1}{2} + \tau} \|u - u^\delta\|_{H^3} ^{\f{1}{2} - \tau} \leq C \delta^\tau.
		\eali\eeq
		For the perturbed profiles, we have
		\beq\bali\nonumber
		(u^\delta_{m, \gamma} - u^\delta)''(y) = m \gamma (\theta^\delta)'(y) \Gamma(\f{y}{\gamma} ) + m \theta^\delta(y) \Gamma'(\f{y}{\gamma} ),
		\eali\eeq
		so that 
		\beq\bali\nonumber
		\|u^\delta_{m, \gamma} - u^\delta\|_{H^{\f{5}{2} - \tau}} \leq& C |m| \left( \|\theta^\delta\|_{H^{\f{3}{2} - \tau}} \|\gamma \Gamma(\f{y}{\gamma}) \|_{L^\infty} + \|\theta^\delta\|_{L^\infty} \|\Gamma'(\f{y}{\gamma}) \|_{H^{\f{1}{2} - \tau}} \right)\\
		\leq& C |m|\delta^\tau + C |m| \gamma^\tau.
		\eali\eeq
		The inequalities for $\theta$ follow similarly, which completes the proof of the lemma.
	\end{proof}
	
	We introduce a horizontal scaling transformation 
	$$\zeta = k x, \quad \nabla_\zeta = (k \pa_\zeta, \pa_y), \quad \Delta_\zeta = k^2 \pa_\zeta^2 + \pa_y^2,$$
	and define a functional $\mathcal{F}$ as
	\beq\bali\nonumber
	\CF(k, \Psi) =& \Delta_\zeta \Psi - \f{G'(\Psi)}{2G(\Psi)} |\nabla_\zeta \Psi|^2 - L(\Psi)\\
	=& k^2 \left(\pa_\zeta^2 \Psi - \f{G'(\Psi)}{2G(\Psi)} |\pa_\zeta \Psi|^2 \right) + \pa_y^2 \Psi - \f{G'(\Psi)}{2G(\Psi)} |\pa_y \Psi|^2 - L(\Psi).
	\eali\eeq
	For the shear flow given by
	\beq\nonumber
	\uu = (u^\delta_{m, \gamma}(y), 0), \quad \theta = \theta^\delta_{l, \eta}(y), \quad \psi^\delta_{m, \gamma}(y) = -\int_0^y u^\delta_{m, \gamma}(z) \dz.
	\eeq
	The functions $G$ and $L$ can be expressed by
	\beq\bali\label{expression of G and L}
	G(\psi^\delta_{m, \gamma}) =& \theta^\delta_{l, \eta}(y),\\
	L(\psi^\delta_{m, \gamma}) =&  -\left( u^\delta_{m, \gamma} \right)^\prime + \f{\left( \theta^\delta_{l, \eta} \right)'}{2 \theta^\delta_{l, \eta}} u^\delta_{m, \gamma}.
	\eali\eeq
	Due to the monotonicity of $u^\delta_{m,\gamma}$, the stream function $\psi^\delta_{m, \gamma}$ is monotonic in both $(-1, 0)$ and $(0, 1)$. Thus, we can define $G$ and $L$ piecewise and extend the definition \eqref{expression of G and L} to a neighborhood of $\psi^\delta_{m, \gamma}$.
	
	\begin{Def}
		Let $\psi^\delta_{m, \gamma}$ be the stream function generated by $u^\delta_{m, \gamma}$ that
		\beq\nonumber
		\psi^\delta_{m, \gamma}(y) = -\int_0^y u^\delta_{m, \gamma}(z) \dz.
		\eeq
		Define functions $\tilde{G} = \mathbf{1}_{y<0}\ \tilde{G}_- + \mathbf{1}_{  y>0}\ \tilde{G}_+$ by 
		\beq\nonumber
		\left\{\bali
		&\tilde{G}_-(\psi^\delta_{m, \gamma}(y)) = \theta^\delta_{l, \eta}(y), \quad y \in [-1, 0),\\
		&\tilde{G}_+(\psi^\delta_{m, \gamma}(y)) = \theta^\delta_{l, \eta}(y), \quad y \in (0, 1],
		\eali\right.
		\eeq
		and  $\tilde{L} = \mathbf{1}_{y<0}\ \tilde{L}_- + \mathbf{1}_{  y>0}\ \tilde{L}_+$ by
		\beq\nonumber
		\left\{\bali
		&\tilde{L}_-(\psi^\delta_{m, \gamma}(y)) = -\left( u^\delta_{m, \gamma} \right)^\prime + \f{\left( \theta^\delta_{l, \eta} \right)'}{2 \theta^\delta_{l, \eta}} u^\delta_{m, \gamma},  &&y \in [-1, 0),\\
		&\tilde{L}_+(\psi^\delta_{m, \gamma}(y)) = -\left( u^\delta_{m, \gamma}  \right)^\prime + \f{\left( \theta^\delta_{l, \eta} \right)'}{2 \theta^\delta_{l, \eta}} u^\delta_{m, \gamma}, &&y \in (0, 1],
		\eali\right.
		\eeq
		Noting that {  $\tilde{L} \in C^2, \tilde{G} \in C^3$}, we denote their extensions to $\mathbb{R}$ by $G$ and $L$ as
		\beq\bali\nonumber
		G = \mathbf{1}_{y<0}\ G_- + \mathbf{1}_{  y>0}\ G_+, \quad L = \mathbf{1}_{y<0}\ L_- + \mathbf{1}_{  y>0}\ L_+,
		\eali\eeq
		such that
		\beq\bali\nonumber
		&G_\pm \big|_{  [\psi^\delta_{m, \gamma}(\pm 1),0]} = \tilde{G}_\pm, \quad &&G_+(z) = G_-(z) = {  \sum_{n=0}^3 \f{1}{n!} G_+^{(n)}(0)z^n}, z \in {  \BR^+}.\\
		&L_\pm \big|_{  [\psi^\delta_{m, \gamma}(\pm 1),0]} = \tilde{L}_\pm, \quad &&L_+(z) = L_-(z) = L_+(0) + L_+'(0) z + \f{1}{2} L_+''(0) z^2, z \in {  \BR^+}.
		\eali\eeq
	\end{Def}
	
	By taking $\psi^\delta_{m, \gamma}$ into $\CF$, we obtain the Fr\'echet derivatives
	\beq\nonumber
	\f{\delta \CF}{\delta k}(k, \psi^\delta_{m, \gamma}) = 0, \quad \f{\delta \CF}{\delta \Psi}(k, \psi^\delta_{m, \gamma}) \psi = k^2 \pa_\zeta^2 \psi - \CH_{m, \gamma, l, \eta}^\delta {  \psi}.
	\eeq
	Here, the associated operator $\mathcal{H}_{m, \gamma, l, \eta}^\delta : H_0^1(-1, 1) \rightarrow H^{-1}(-1, 1)$ is defined by
	\beq\nonumber
	\CH_{m, \gamma, l, \eta}^\delta \psi = - \theta^\delta_{l,\eta} \pa_y \left( \f{1}{\theta^\delta_{l,\eta}} \pa_y \psi \right) + \f{ \theta^\delta_{l,\eta} }{ u^\delta_{m,\gamma} } \left( \f{ (u^\delta_{m,\gamma})' }{ \theta^\delta_{l,\eta} } \right)'{  \psi}.
	\eeq
	From Section \ref{section 2}, we know that the minimal eigenvalue of $\mathcal{H}_{m, \gamma, l, \eta}^\delta$ can be shifted to a negative value by varying $m$ and $l$. The following lemma shows that a bifurcation branch emanates from $\psi^\delta_{m, \gamma}$ if $\mathcal{H}_{m, \gamma, l, \eta}^\delta$ possesses a negative eigenvalue.
	
	\begin{lem}\label{lemma of bifurcation}
		If the operator $\CH_{m, \gamma, l, \eta}^\delta$ has a negative eigenvalue $-k_0^2$, then there exists $\alpha_0 > 0$ and continuous functions $\Psi(\alpha): (-\alpha_0, \alpha_0) \rightarrow H^2$ and $k(\alpha): (-\alpha_0, \alpha_0) \rightarrow \mathbb{R}$ satisfying
		\beq\bali\nonumber
		\CF(k(\alpha), \Psi(\alpha)) = 0,
		\eali\eeq
		where the leading order terms are given by
		$$k(0) = k_0, \quad \Psi(\alpha) = \psi^\delta_{m, \gamma} + \alpha \phi_0(y) \cos(\zeta) + o(\alpha), $$
		and the solution satisfies
		$$\| \Psi(\alpha) - \psi^\delta_{m, \gamma} \|_{H^4( \BT_{\f{2\pi}{k(\alpha)}} \times (-1, 1) )} \leq C \alpha.$$
	\end{lem}
	
	\begin{proof}
		Denote $\Psi_0 = \psi^\delta_{m, \gamma}$, and consider the perturbation $\Psi = \Psi_0 + \psi$. Define $F$ as
		\beq\bali\label{expression of F}
		F(k, \psi) =&  \CF(k, \Psi_0 + {  \psi}) - \CF(k, \Psi_0)\\
		=&  k^2 \left({  \pa_\zeta^2\psi} +  \f{G'(\Psi_0)}{2G(\Psi_0)} |\pa_\zeta \Psi_0 |^2 - \f{G'(\Psi_0 + \psi)}{2G(\Psi_0 + \psi)} |\pa_\zeta \Psi_0 + \pa_\zeta \psi|^2 \right) \\
		&+ \pa_y^2 \psi + \f{G'(\Psi_0)}{2G(\Psi_0)} |{  \pa_y \Psi_0}|^2  - \f{G'(\Psi_0 + \psi)}{2G(\Psi_0 + \psi)} |\pa_y \Psi_0 + \pa_y \psi|^2  + L(\Psi_0) - L(\Psi).
		\eali\eeq
		The Fr\'echet derivatives of $F$ are
		\beq\bali\label{Frechet derivatives of CF}
		&\f{\delta \CF}{\delta k}(k, 0) 
		= 2k \left(\pa_\zeta^2 \Psi_0 - \f{G'(\Psi_s)}{2G(\Psi_0)} |\pa_\zeta \Psi_0|^2 \right),\\
		&\f{\delta F}{\delta \Psi}(k, 0) \psi 
		= \Delta_\zeta \psi 
		- \f{G'(\Psi_0)}{G(\Psi_0)} \nabla_\zeta \Psi_0 \cdot \nabla_\zeta  \psi 
		- \left(\left(\f{G'}{2G}\right)'(\Psi) |\nabla_\zeta \Psi_0|^2 + L'(\Psi_0) \right) \psi, \\
		&\f{\delta^2 F}{\delta k \delta \Psi}(k, 0) \psi 
		= 2k \left(\pa_\zeta^2 \psi - \f{G'(\Psi_s)}{G(\Psi_0)} \pa_\zeta \Psi_0 \pa_\zeta \psi - \left(\f{G'}{2G}\right)'(\Psi_0) |\pa_\zeta \Psi_0|^2 \right).
		\eali\eeq
		Here, $G \in C^3$ and $L \in C^2$, since $u_{m, \gamma}^\delta, \theta_{l, \eta}^\delta \in C^3(-1,1) \cap C^\infty(-\delta, \delta)$. So the Fr\'echet derivatives $\f{\delta F}{\delta \Psi}$ and $\f{\delta^2 F}{\delta k \delta \Psi}$ are continuous in a neighborhood of the origin in $H^3(\BT_{2\pi} \times [-1, 1])$. Define the spaces
		\beq\bali\nonumber
		X =& \{ \psi(\zeta, y) \in H^3(\BT_{2\pi} \times [-1, 1]),\  \psi(\zeta, -1) = \psi(\zeta, 1) = 0,\ \psi(\zeta, y) = \psi(-\zeta, y) \},\\
		Y =& \{ \psi(\zeta, y) \in H^1(\BT_{2\pi} \times [-1, 1]),\  \psi(\zeta, y) = \psi(-\zeta, y) \},
		\eali\eeq
		and let $\Omega = X \cap \{ \|\psi\|_{H^3} \leq \delta_0\}$ for some $\delta_0 \ll 1$ such that $F$ is twice-order differential continuous in $\Omega$. We consider the mapping
		\beq\bali\nonumber
		F(k,  \psi) : \BR \times \Omega \rightarrow Y.
		\eali\eeq
		From \eqref{Frechet derivatives of CF}, the linearized operator of $F$ at $(k_0, 0)$ is
		\beq\bali\nonumber
		\CL = \f{\delta F}{\delta \Psi}(k_0, 0) = k_0^2 \pa_\zeta^2 - \CH_{m, \gamma, l, \eta}^\delta.
		\eali\eeq
		Let $\phi_0 \in H^3(-1,1)$ be the eigenfunction of $\CH_{m, \gamma, l, \eta}^\delta$ corresponding to the eigenvalue $-k_0^2$. We find that $\psi_0(\zeta, y) = \cos(\zeta) \phi_0(y)$ spans the kernel of $\CL$, such that
		\beq\bali\nonumber
		\text{N}(\CL) = \Span\{ \psi_0\}, \quad Y = \text{R}(\CL) \oplus \Span\{ \psi_0\}.
		\eali\eeq
		Also, we find that the transversality condition holds:
		\beq\bali\nonumber
		\f{\delta^2 F}{\delta k \delta \Psi}(k_0, 0) \psi_0
		= 2k_0 \pa_\zeta^2 \psi_0 = -2k_0 \psi_0 \notin \text{R}(\CL).
		\eali\eeq
		Applying the Crandall-Rabinowitz local bifurcation theorem, there exists $\alpha_0 > 0$ and continuous functions $k(\alpha): (-\alpha_0, \alpha_0) \rightarrow \BR$ and $\psi(\alpha): (-\alpha_0, \alpha_0) \rightarrow X$ such that
		\beq\bali\nonumber
		F(k(\alpha), \psi(\alpha)) = 0,
		\eali\eeq
		where the leading order term is $\psi(\alpha) = \alpha \phi_0(y) \cos(\zeta) + o(\alpha)$, and the perturbed stream function is
		\beq\nonumber
		\Psi(\alpha) = \Psi_0 + \alpha \phi_0(y) \cos(\zeta) + o(\alpha).
		\eeq
		Based on equation \eqref{equation of CF} and the regularity $G \in C^3$, $L \in C^2$, we can bootstrap the regularity of $\Psi(\alpha)$ to obtain
		\beq\bali\nonumber
		\| \Psi(\alpha) - \Psi_0 \|_{H^4} = O(\alpha).
		\eali\eeq
	\end{proof}
	
	At the end of this section, we present and complete the proof of Theorem \ref{Thm: travel-wave}.
	
	\begin{proof}[\textbf{Proof of the Theorem \ref{Thm: travel-wave}}]
		Recall that the stream function $\Psi$ satisfies
		\beq\label{equation of CF}
		\CF(k,  \Psi) = \Delta_\zeta  \Psi - \f{G'(\Psi)}{2G(\Psi)} |\nabla_\zeta \Psi|^2 - L(\Psi) = 0, 
		\eeq
		which also yields a solution to the Euler equation \eqref{incompressible Euler system}. 
		We need to construct perturbed solutions to $\CF = 0$ corresponding to $k = 1$, ensuring $\Psi$ is $2\pi$-periodic in $x$. From Lemma \ref{Lemma: Lambda(lambda)}, there exists $\Lambda = \Lambda(-1)$ such that the Rayleigh operator admits a negative eigenvalue $\lambda = -1$.
		For any fixed $l \in \BR$, we can find $m_1$ and $m_2$ such that
		\beq\bali\nonumber
		\Lambda_{m_1, l} > \Lambda(-1) > \Lambda_{m_2, l}.
		\eali\eeq
		By the strict monotonicity of $\Lambda(\lambda)$ and Lemma \ref{lemma for smallness of gamma and eta}, there exist $\eta_0>0$ and $\gamma_0>0$ such that for all $\eta \in (0, \eta_0)$ and $\gamma \in (0, \gamma_0)$, the negative eigenvalue of the associated Rayleigh operator satisfies
		\beq\bali\nonumber
		\lambda_{m_1,\gamma,l,\eta} < -1 < \lambda_{m_2,\gamma,l,\eta}.
		\eali\eeq
		Applying the Lemma \ref{lemma of bifurcation}, we obtain $r_0 > 0$ such that for each $m \in (m_2, m_1)$ and $r \in (0, r_0)$, there exist $\alpha = \alpha(m,r)$ and a perturbed stream function $\Psi(\alpha)$ satisfying
		\beq\bali\nonumber
		\| (u_\alpha, v_\alpha) - (u^\delta_{m, \gamma},0)\|_{H^3(\BT_{\f{2\pi}{k(\alpha)}} \times (-1,1))} + \|\theta_\alpha - \theta^\delta_{l,\eta}\|_{H^2(\BT_{\f{2\pi}{k(\alpha)}} \times (-1,1))} = r.
		\eali\eeq
		Since the wave number $k(\alpha(m,r))$ depends continuously on $m$, the Intermediate Value Theorem guarantees the existence of $m_0 \in (m_2, m_1)$ such that $k(\alpha(m_0,r)) = 1$. We arrive at
		\beq\bali\nonumber
		\| (u_\alpha, v_\alpha) - (u^\delta_{m_0, \gamma},0)\|_{H^3(\BT_{2\pi} \times (-1,1))} + \|\theta_\alpha - \theta^\delta_{l,\eta}\|_{H^2(\BT_{2\pi} \times (-1,1))} = r.
		\eali\eeq
		Combining this result with Lemma \ref{lemma of distance between u,theta and u^delta, theta^delta}, we choose $r = c_0 {  \epsilon}$ and require $\delta, \gamma, \eta \leq c_0 \epsilon$ for a sufficiently small constant $c_0$. We thus obtain
		\beq\bali\nonumber
		\| \uu_\alpha - (u(y),0)\|_{H^{\f{5}{2} - \tau}} + \|\theta_\alpha - \theta\|_{H^{\f{3}{2} - \tau}} \leq \epsilon.
		\eali\eeq
		This completes the proof.
	\end{proof}

	\section{Non-existence of the traveling wave solutions in $H^{\f{3}{2}+\tau}$}
	In this section, we establish the non-existence of traveling wave solutions in high-order Sobolev spaces. Based on the argument in the subsection \ref{sebsection of idea of proof}, Theorem \ref{Theorem of nonlinear stability} follows directly from the following lemma.
	
	\begin{prop}\label{Proposition of non-existence of travelling wave solutions}
		Let $u,\theta$ be such that the distorted Rayleigh operator $u(y) \text{Id} - (\f{u'}{\theta})'(y) \tilde{\Delta}^{-1}$ has no eigenvalues or embedding eigenvalues. For any $\tau > 0$, there exists $\varepsilon_0 > 0$ such that if $u_s, v_s, \theta_s$  and $c \in \BR$ satisfying
		\beq\label{eq0.1}
		\| (u_s, v_s) - (u(y), 0)\|_{H^{\f{5}{2}+\tau}} + \| \theta_s - \theta(y) \|_{H^{\f{3}{2}+\tau}} \leq \epsilon_0,
		\eeq
		is a solution of \eqref{eq0.2} with the boundary condition $v_s (x, 1) = v_s (x, -1) =0$, then $v_s \equiv 0$.
	\end{prop}
	
	We first establish a uniform resolvent estimate.
	
	\begin{lem}\label{resolvent estimate}
		Let $u,\theta$ be such that the distorted Rayleigh operator $u(y) \text{Id} - (\f{u'}{\theta})'(y) \tilde{\Delta}^{-1}$ has no eigenvalues or embedded eigenvalues. Then for any $\tau > 0$, there is $\epsilon_0 > 0$ such that for all $u_s, \theta_s$ satisfying
		$$ \| u_s(x, y) - u(y) \|_{H^{\f{5}{2}+\tau}(\BT_{2\pi} \times (-1, 1))} + \| \theta_s(x, y) - \theta(y) \|_{H^{\f{3}{2}+\tau}(\BT_{2\pi} \times (-1, 1))} \leq \epsilon_0, $$
		any solution {  $v\in H^2\cap H_0^1(\BT_{2\pi} \times (-1, 1))$} to the equation
		\beq \label{eq0.4}
		(u_s - c) \nabla \cdot \left( \f{1}{\theta_s} \nabla v \right) - v \left( \f{u'}{\theta}\right)' = F, 
		\eeq 
		with Dirichlet boundary condition $v (x, 1) = v (x, -1) =0$ and all $c \in \BC$ satisfies
		\beq\label{eq: H^1 bound}
		\|\nabla v\|_{L^2} \leq C \|F\|_{H^{\f{1}{2} + \tau}}.
		\eeq
		Here, the constant C is independent of $c$.
	\end{lem}
	{ 
\begin{proof}
        It suffices to consider $0<\tau<\f{1}{4}$. If $|c|\geq M$ for some sufficiently large $M$, then \eqref{eq: H^1 bound} holds. Dividing \eqref{eq0.4} by $u_s-c$, multiplying by $\bar v$, and integrating by parts gives 
        \beq\bali\nonumber
        \int_{\BT_{2\pi} \times (-1, 1)} \f{1}{\theta_s} |\nabla v|^2 \,\dx\dy =& \left|\int_{\BT_{2\pi} \times (-1, 1)} \f{|v|^2 \left( \f{u'}{\theta}\right)' + F \bar{v}}{u_s - c} \,\dx\dy\right|\\
        \leq& \f{C}{M} \left( \|v\|_{L^2}^2 + \|F\|_{L^2} \|v\|_{L^2} \right)\\
        \leq& \f{C}{M} \left( \|\nabla v\|_{L^2}^2 + \|F\|_{L^2}^2 \right). 
        \eali\eeq
        For sufficiently large $M$, the first term on the right can be absorbed into the left-hand side, yielding \eqref{eq: H^1 bound}.
        
        Now, we focus on $|c|\leq M$. 
        We prove the lemma by a contradiction argument. Suppose that there exists some $\tau$ such that for any $\epsilon_0$, the estimate \eqref{eq: H^1 bound} does not hold, namely, for any $k>0$ large enough, there exists a sequence  $\{u_k, v_k, \theta_k, F_k, c_k\}$ satisfying the equation
        \beq\label{equation of u_k, v_k, F_k}
        (u_k - c_k) \nabla \cdot \left( \f{1}{\theta_k} \nabla v_k \right) - v_k \left( \f{u'}{\theta}\right)' = F_k,
        \eeq
        subject to the bounds
        \beq\bali\nonumber
        \| u_k - u(y)\|_{H^{\f{5}{2}+\tau}} +& \| \theta_k - \theta(y) \|_{H^{\f{3}{2}+\tau}} \leq \f{1}{k},\\
        \| \nabla v_k\|_{L^2} = 1,& \quad \|F_k\|_{H^{\f{1}{2} + \tau}} \leq \f{1}{k},
        \eali\eeq
        Passing to a subsequence, we obtain
        \beq\bali\label{expressions of limit convergence}
        c_k \rightarrow c_\infty, \quad v_k \rightarrow v_\infty \text{ in } H^{\f{1}{2}+\tau}, \quad v_k \rightharpoonup v_\infty \text{ in } H^{1} \text{ and } F_k \rightarrow 0 \text{ in } {H^{\f{1}{2} + \tau}}.
        \eali\eeq

        We now distinguish the two cases $c_\infty\notin\operatorname{Ran}u$ and $c_\infty\in\operatorname{Ran}u$, where the range includes the boundary values $u(\pm1)$.
        Before using the divided equation in either case, we explain the role of the assumption $v_k\in H^2$.
        Since $\theta_k\in H^{\f32+\tau}$ is uniformly positive, Sobolev multiplication gives
        \[
        \nabla\cdot\left(\f1{\theta_k}\nabla v_k\right)
        =\f1{\theta_k}\Delta v_k-\f{\nabla\theta_k}{\theta_k^2}\cdot\nabla v_k\in L^2.
        \]
        Moreover, $\partial_yu_k$ is bounded away from zero for large $k$, so that $\{u_k=c_k\}$ has two-dimensional measure zero even when $c_k$ is real.
        Equation \eqref{equation of u_k, v_k, F_k} therefore implies the almost-everywhere identity
        \beq\label{eq: division form of v_k}
        \f{v_k\left(\f{u'}{\theta}\right)'+F_k}{u_k-c_k}
        =\nabla\cdot\left(\f1{\theta_k}\nabla v_k\right)\in L^2.
        \eeq
        Thus division by $u_k-c_k$ is legitimate as an $L^2$ identity, including for real $c_k$; no pointwise value of the quotient on $\{u_k=c_k\}$ is needed. In particular, for almost every $x$, if $u_k(x,y_k(x))=c_k\in\BR$, then
        \[
        v_k(x,y_k(x))\left(\f{u'}{\theta}\right)'(y_k(x))+F_k(x,y_k(x))=0.
        \]
        Indeed, the numerator belongs to $H_y^{\f12+\tau}\hookrightarrow C_y^0$ for almost every $x$. A nonzero value at the critical point, together with the simple zero of $u_k-c_k$, would contradict \eqref{eq: division form of v_k}.
        Notice also that, in the application to Proposition~\ref{Proposition of non-existence of travelling wave solutions}, $v=v_s$ actually belongs to $H^{\f52+\tau}$ by \eqref{eq0.1}, so the $H^2$ assumption is satisfied with room to spare.

        Testing \eqref{eq: division form of v_k} against $\bar\psi$, for $\psi\in H_0^1(\BT_{2\pi}\times(-1,1))$, yields
        \beq\label{eq: weak formulation of v_k}
        \int_{\BT_{2\pi}\times(-1,1)}\f1{\theta_k}\nabla v_k\cdot\nabla\bar\psi\,\dx\dy
        +\int_{\BT_{2\pi}\times(-1,1)}
        \f{v_k\left(\f{u'}{\theta}\right)'+F_k}{u_k-c_k}\bar\psi\,\dx\dy=0.
        \eeq

        If $c_\infty\notin\operatorname{Ran}u$, then $|u_k-c_k|$ is uniformly bounded away from zero for large $k$. Taking $\psi=v_k$ in \eqref{eq: weak formulation of v_k} gives
        \[
        1=\|\nabla v_k\|_{L^2}^2
        \le C\left(\|v_k\|_{L^2}^2+\|F_k\|_{L^2}\|v_k\|_{L^2}\right).
        \]
        Hence $v_\infty\ne0$. Passing to the limit in \eqref{eq: weak formulation of v_k}, we obtain
        \beq\label{eq: weak formulation of v_infty 1}
        \int\f1\theta\nabla v_\infty\cdot\nabla\bar\psi\,\dx\dy
        +\int\f{v_\infty\left(\f{u'}\theta\right)'}{u-c_\infty}\bar\psi\,\dx\dy=0.
        \eeq
        Consequently,
        \beq\label{eq: strong formulation of v_infty}
        \nabla\cdot\left(\f1\theta\nabla v_\infty\right)
        =\f{v_\infty\left(\f{u'}\theta\right)'}{u-c_\infty}\in L^2.
        \eeq
        Elliptic regularity implies $v_\infty\in H^2\cap H_0^1$, so $c_\infty$ is an eigenvalue of the distorted Rayleigh operator, contradicting the spectral assumption.

        It remains to consider $c_\infty\in\operatorname{Ran}u$. Write $u(y_\infty)=c_\infty$, with $y_\infty\in[-1,1]$. We establish the uniform estimate
        \beq\label{eq: bound of nabla v_k}
        \|\nabla v_k\|_{L^2} \leq C \left(\|v_k\|_{H^{\f12+\tau}} + \|F_k\|_{H^{\f12+\tau}}\right),
        \eeq
        with $C$ independent of $k$ and $c_k$, including complex $c_k$ approaching the real axis.
        Extend $u_k$ monotonically to $\BT_{2\pi}\times\BR$ so that
        \[
        0<a<\partial_yu_k<b,\qquad
        \|\partial_yu_k\|_{C^{0,\f12+\f{3\tau}4}}\le C.
        \]
        Extend $v_k,F_k$ by a fixed bounded extension in the $y$ variable, extend $\left(\f{u'}\theta\right)'$ smoothly, and extend $\psi$ by zero. Use the original change of variables
        \[
        z_k=u_k(x,y),\qquad y=Y_k(x,z_k),\qquad
        J_k(x,z_k)=\partial_yu_k(x,Y_k(x,z_k)),
        \]
        and denote composition with $Y_k$ by a tilde, as before. Then
        \beq\label{eq: expression of v_k in z}
        \int_{\BT_{2\pi}\times(-1,1)}
        \f{v_k\left(\f{u'}\theta\right)'+F_k}{u_k-c_k}\bar\psi\,\dx\dy
        =\int_{\BT_{2\pi}\times\BR}
        \f{\tilde v_k\widetilde{\left(\f{u'}\theta\right)'}+\tilde F_k}{z_k-c_k}
        \f{\overline{\tilde\psi}}{J_k}\,\dx\dz_k.
        \eeq
        For real $c_k$ the right-hand side can equivalently be interpreted as a principal value in $z_k$, since the quotient on the left belongs to $L^2$.
        The Hilbert transform is bounded on $H^{\f{1+\tau}2}(\BR)\hookrightarrow C_b^0(\BR)$. In either open half-plane, the Cauchy transform is the Poisson extension of its boundary value, which is a linear combination of the Hilbert transform and the original function. Thus, uniformly for real and complex $c_k$,
        \beq\bali\label{eq: bound of Hilbert transform}
        &\left|\int_\BR
        \f{\tilde v_k\widetilde{\left(\f{u'}\theta\right)'}+\tilde F_k}{z_k-c_k}
        \f{\overline{\tilde\psi}}{J_k}\,\dz_k\right|\\
        &\quad\le C
        \left\|\tilde v_k\widetilde{\left(\f{u'}\theta\right)'}+\tilde F_k\right\|_{H_{z_k}^{\f{1+\tau}2}}
        \left\|\f{\tilde\psi}{J_k}\right\|_{H_{z_k}^{\f{1+\tau}2}}.
        \eali\eeq
        Here we used that $H^{\f{1+\tau}2}(\BR)$ is an algebra. The Gagliardo characterization of fractional Sobolev spaces and the uniform bi-Lipschitz bounds for $Y_k$ imply boundedness of composition on $H^{\f{1+\tau}2}$. Also, multiplication by $J_k^{-1}$ is bounded on this space, since $\f12+\f{3\tau}4>\f{1+\tau}2$. Hence
        \[
        \left\|\tilde v_k\widetilde{\left(\f{u'}\theta\right)'}+\tilde F_k\right\|_{L_x^2H_{z_k}^{\f{1+\tau}2}}
        \le C\left(\|v_k\|_{H^{\f12+\tau}}+\|F_k\|_{H^{\f12+\tau}}\right),
        \qquad
        \left\|\f{\tilde\psi}{J_k}\right\|_{L_x^2H_{z_k}^{\f{1+\tau}2}}
        \le C\|\psi\|_{H^1}.
        \]
        Integrating \eqref{eq: bound of Hilbert transform} in $x$ and applying Cauchy--Schwarz, we conclude that
        \[
        \left|\int_{\BT_{2\pi}\times(-1,1)}
        \f{v_k\left(\f{u'}\theta\right)'+F_k}{u_k-c_k}\bar\psi\,\dx\dy\right|
        \le C\left(\|v_k\|_{H^{\f12+\tau}}+\|F_k\|_{H^{\f12+\tau}}\right)\|\psi\|_{H^1}.
        \]
        Taking $\psi=v_k$ in \eqref{eq: weak formulation of v_k} and using uniform positivity of $\theta_k^{-1}$ and Poincare's inequality proves \eqref{eq: bound of nabla v_k}. It follows from \eqref{expressions of limit convergence} that
        \[
        \|v_\infty\|_{H^{\f12+\tau}}\ge C^{-1}>0.
        \]

        We now pass to the limit at the critical layer. If there is a subsequence $c_{k} \to c_\infty$ that $c_{k} = r_{k} + i \eta_{k}$ with $\eta_{k} \neq 0$. We may assume that $\eta_k>0$ along a subsequence; the lower half-plane is treated in the same way, with the opposite sign in the imaginary boundary term. Denote
        $$g_k(x,y) = \f{v_k\left(\f{u'}{\theta}\right)'+F_k}{\pa_y u_k}\bar{\psi}, \quad G_k(z) = \int_{\BT_{2\pi}} g_k(x,Y_k(x,z)) \dx.$$
        By applying the compatible extension of $u_k$ such that $Y_k \to Y_\infty$ and $\pa_z Y_k \to \pa_z Y_\infty$ uniformly in any compact set of $\BR$, we get, for any fixed $f\in L_x^2H_y^{\f{1+\tau}{2}}$ and $R>0$, we have
        $$f(x, Y_k(x,z)) \to f(x, Y_\infty(x,z)) \text{ in } L_x^2H_z^{\f{1+\tau}{2}}(-R,R).$$
        This follows by smooth approximation and the uniform composition bounds. The compatible extensions can be chosen so that the Jacobians also converge in $C^{0,\f12+\f{3\tau}4}$ on compact sets. The multiplication bounds used above, together with the strong convergence in \eqref{expressions of limit convergence}, therefore justify the convergence of $G_k$ below.
        Since $\psi$ is extended by zero outside $[-1,1]$, we can select $R$ large enough such that $\text{supp}\ g_k\subset (-R,R)$. So that
        $$G_k(z) \to G_\infty(z) \text{ in } H_z^{\f{1+\tau}{2}}(\BR).$$
        By the property of Hilbert transform, we have
        \beq\bali\nonumber
        \big|&\int_\BR \f{G_k(z)}{z-c_k} \dz - P.V. \int_\BR \f{G_\infty(z)}{z-c_\infty} \dz - i\pi G_\infty(c_\infty)\big|\\
        \leq& \big|\int_\BR \f{G_k(z)-G_\infty(z)}{z-c_k} \dz \big| + \big| \int_\BR \f{G_\infty(z)}{z-c_k} \dz - P.V. \int_\BR \f{G_\infty(z)}{z-c_\infty} \dz - i\pi G_\infty(c_\infty)\big|\\
        \leq& C \|G_k - G_\infty\|_{H_z^{\f{1+\tau}{2}}} + C \|G_\infty\|_{H_z^{\f{1+\tau}{2}}} |c_k - c_\infty|^{\tau/4} \to 0.
        \eali\eeq
        Hence we find that
        \beq\bali\nonumber
        \lim_{k \to \infty} \int_{\BT_{2\pi} \times (-1, 1)} \f{v_k \left( \f{u'}{\theta}\right)' + F_k}{u_k - c_k} \bar{\psi} \,\dx \dy =& \lim_{k \to \infty} \int_\BR \f{G_k(z)}{z-c_k} \dz\\
        =& P.V. \int_\BR \f{G_\infty(z)}{z-c_\infty} \dz + i\pi G_\infty(c_\infty).
        \eali\eeq
        Let $k \to \infty$, we get
        \beq\bali\nonumber
        \left\langle \f{1}{\theta} \nabla v_\infty, \nabla \psi \right\rangle 
        + P.V. \left\langle \f{v_\infty \left( \f{u'}{\theta}\right)' }{u - c_\infty}, \psi \right\rangle 
        + i\pi \f{\left( \f{u'}{\theta}\right)'(y_\infty) }{ u'(y_\infty) }\int_{\BT_{2\pi}} v_\infty(x, y_\infty) \bar{\psi}(x, y_\infty) \dx = 0.
        \eali\eeq
        At $y_\infty=\pm1$, the trace term vanishes by the boundary conditions. Taking $\psi=v_\infty$ and the imaginary part yields
        $$\left( \f{u'}{\theta}\right)'(y_\infty) v_\infty(x, y_\infty) \equiv 0.$$ 
        In case $\left( \f{u'}{\theta}\right)'(y_\infty) = 0$,  since $\left( \f{u'}{\theta}\right)'\in C^1$, we find that $ \f{1}{u-c_\infty}\left( \f{u'}{\theta}\right)'$ is bounded. So the \eqref{eq: strong formulation of v_infty} holds.
        The $v_\infty$ is a strong solution to \eqref{eq: strong formulation of v_infty} and thus $c_\infty$ is an eigenvalue of the distorted Rayleigh operator $u(y) \text{Id} - (\f{u'}{\theta})'(y) \tilde{\Delta}^{-1}$, which leads to a contradiction. In case $v_\infty(x, y_\infty) \equiv 0$, we use the bounds $|u(y)-c_\infty| \geq c|y-y_\infty|$ and the Hardy inequality to obtain
        $$\left\| \f{v_\infty }{u - c_\infty} \left( \f{u'}{\theta}\right)' \right\|_{L^2} \leq C \|\left( \f{u'}{\theta}\right)'\|_{L^\infty} \left\| \f{v_\infty}{y - y_\infty} \right\|_{L^2} \leq C \|\pa_y v_\infty\|_{L^2} \leq C.$$
        So the \eqref{eq: strong formulation of v_infty} holds, $c_\infty$ is an eigenvalue of the distorted Rayleigh operator $u(y) \text{Id} - (\f{u'}{\theta})'(y) \tilde{\Delta}^{-1}$ and also leads to a contradiction.

        If all the $c_k$ are real, by the same argument, we let $c_k \to c_\infty$ and get
        \beq\bali\nonumber
        \left\langle \f{1}{\theta} \nabla v_\infty, \nabla \psi \right\rangle 
        + \text{P.V.}\left\langle \f{v_\infty \left( \f{u'}{\theta}\right)' }{u - c_\infty}, \psi \right\rangle = 0.
        \eali\eeq
        Then we show that $v_\infty(x,y_\infty) \left( \f{u'}{\theta}\right)'(y_\infty) \equiv 0$. If $y_\infty = \pm 1$, then we directly have $v_\infty(x,y_\infty)=0$. For $y_\infty \in (-1,1)$, there is $K>0$ such that for all $k>K$, there exists $y_k(x)$ for all $x\in \BT_{2\pi}$ that $u_k(x,y_k(x))=c_k$. By the $L^2$ division and trace argument preceding \eqref{eq: weak formulation of v_k}, we have $v_k(x,y_k) \left( \f{u'}{\theta}\right)'(y_k) + F_k(x,y_k(x))= 0$. We estimate as follows:
        \beq\bali\nonumber
        \| & v_\infty(x,y_\infty) \left( \f{u'}{\theta}\right)'(y_\infty) - v_k(x,y_k) \left( \f{u'}{\theta}\right)'(y_k) - F_k(x,y_k(x)) \|_{L_x^2}\\
        \leq& \|v_\infty(x,y_k(x)) \left( \f{u'}{\theta}\right)'(y_k) - v_k(x,y_k) \left( \f{u'}{\theta}\right)'(y_k) - F_k(x,y_k(x)) \|_{L_x^2}\\
        &+ \|v_\infty(x,y_\infty) \left( \f{u'}{\theta}\right)'(y_\infty) - v_\infty(x,y_k(x)) \left( \f{u'}{\theta}\right)'(y_k) \|_{L_x^2}\\
        \leq& C \|v_\infty \left( \f{u'}{\theta}\right)' - v_k\left( \f{u'}{\theta}\right)' - F_k\|_{L_x^2H_y^{\f{1+\tau}{2}}} + C \|\pa_y\left(v_\infty \left( \f{u'}{\theta}\right)'\right)\|_{L^2} \|y_k(x) - y_\infty\|_{L_x^\infty}^\f{1}{2} \to 0.
        \eali\eeq
        So that $ \|v_\infty(x,y_\infty) \left( \f{u'}{\theta}\right)'(y_\infty)\|_{L_x^2} = 0$ and therefore $v_\infty(x,y_\infty) \left( \f{u'}{\theta}\right)'(y_\infty) \equiv 0$. We deduce that
        $$ \f{v_\infty \left( \f{u'}{\theta}\right)' }{u - c_\infty} \in L^2,$$ 
        and then $v_\infty$ is a strong solution to \eqref{eq: strong formulation of v_infty} and therefore $c_\infty$ is an eigenvalue of the distorted Rayleigh operator $u(y) \text{Id} - (\f{u'}{\theta})'(y) \tilde{\Delta}^{-1}$, which leads to a contradiction to the spectral assumption.
    \end{proof}
    }

	Then we finish the proof of the proposition \ref{Proposition of non-existence of travelling wave solutions}.
	
	\begin{proof}[\textbf{Proof of Proposition \ref{Proposition of non-existence of travelling wave solutions}}]
		Suppose that $(u_s, v_s)$ and $\theta_s$ is a solution to the equation \eqref{eq0.2} with $c \in \BR$ satisfying condition \eqref{eq0.1}.
		We rewrite the equation by treating the perturbation coefficient of $v_s$ as a forcing term on the right-hand side:
		\beq\label{eq1 in proof of prop 3.1}
		(u_s - c) \nabla \cdot \left( \f{1}{\theta_s} \nabla v_s \right) - v_s \left( \f{u'}{\theta} \right)' = v_s \left( \nabla \cdot \left( \f{\nabla u_s}{\theta_s} \right) -\left( \f{u'}{\theta} \right)' \right).
		\eeq
		It suffices to establish the following interpolation inequality. Let $p,q$ satisfy
		$$\f{1}{q} = \f{1}{2} - \f{\tau}{4}, \quad \f{1}{p} = \f{1}{2} - \f{1}{q}.$$ 
		Then, by the Sobolev embedding theorem, we have
		\beq\bali\nonumber
		\|vf\|_{L^2} \leq C \|v\|_{L^p} \|f\|_{L^q} \leq C \|v\|_{H^1} \|f\|_{H^\f{\tau}{2}},
		\eali\eeq
		and
		\beq\bali\nonumber
		\|vf\|_{\dot{H}^1} \leq C\left(\|\nabla v\|_{L^2} \|f\|_{L^\infty} + \|v\|_{L^p} \|\nabla f\|_{L^q}\right) \leq C \|v\|_{H^1} \|f\|_{H^\f{1 + \tau}{2}}.
		\eali\eeq
		Using Calder\'on's complex interpolation method, we obtain
		$$\|vf\|_{H^\f{1+\tau}{2}} \leq C \|v\|_{H^1} \|f\|_{H^{\f{1}{2} + \tau}}.$$
		Applying Lemma \ref{resolvent estimate} to \eqref{eq1 in proof of prop 3.1}, we obtain
		\beq\bali\nonumber
		\|\nabla v\|_{L^2} \leq& C \left\| v \left( \nabla \cdot \left( \f{\nabla u_s}{\theta_s} \right) -\left( \f{u'}{\theta} \right)' \right) \right\|_{H^{\f{1 + \tau}{2}}} 
		\leq C\| \nabla v \|_{L^2} 
		\left\| \nabla \cdot \left( \f{\nabla u_s}{\theta_s} \right) -\left( \f{u'}{\theta} \right)'  \right\|_{H^{\f{1}{2} + \tau}} \\
		\leq& C\| \nabla v \|_{L^2} \left( \| u_s- u(y) \|_{H^{\f{5}{2}+\tau}} + \| \theta_s - \theta(y) \|_{H^{\f{3}{2}+\tau}} \right)\\
		\leq& C \epsilon_0 \| \nabla v \|_{L^2}.
		\eali\eeq
		Let $\epsilon_0$ be sufficiently small. We conclude that $\|\nabla v\|_{L^2} = 0$, and therefore $ v \equiv 0$.
	\end{proof}

    \noindent\textbf{Acknowledgments.}
    We are grateful to the anonymous referee for a careful reading of the manuscript and for valuable comments and suggestions, which contributed substantially to the revision and improvement of this paper.

	\appendix
	
	\section{Appendix}
	
	The operator $\CH_{u,\theta}$ leads us to the distorted Rayleigh equation
	\beq\label{distorted Rayleigh equation}
	- \theta \pa_y \left(\f{1}{\theta} \pa_y \Psi\right) + \f{\theta}{u - c} \left(\f{u'}{\theta}\right)' \Psi = \lambda \Psi,
	\eeq
	with the Dirichlet boundary condition $\Psi(-1) = \Psi(1) = 0$.
	
	\begin{prop}[\cite{zhao2025inviscid}(Proposition 3.2)]\label{solution of phi = u phi_1}
		Assume that $u'(y), \theta(y) \geq c_0 > 0$. Let $c = c_r + i c_i$, and let $y'$ be the critical point such that $u(y') = c_r$. For any $\lambda \in \BR$, the distorted Rayleigh equation \eqref{distorted Rayleigh equation} has a regular solution
		\beq\nonumber
		\phi(y) = \left(u(y) - c\right) \phi_1(y),
		\eeq
		where $\phi_1$ satisfies the integral equation
		\beq\nonumber
		\phi_1(y) = 1 - \lambda \int_{y'}^y \f{\theta(z')}{\left(u(z') - c\right)^2} \int_{y'}^{z'} \f{\left(u(z'') - c\right)^2}{\theta(z'')} \phi_1(z'') \dz''\dz'.
		\eeq
		
		Moreover, for $c \in \BR$ and $\lambda <0$, there hold
		\begin{align}\nonumber
			\phi_1(y) \geq 1, \quad (y - y') \pa_y \phi_1(y) \geq 0,
		\end{align}
		\begin{align}\nonumber
			C^{-1} \min\{-\lambda|y-y'|,\sqrt{-\lambda}\} \leq \f{|\pa_y \phi_1(y)|}{\phi_1(y)} \leq C \min\{-\lambda|y-y'|,\sqrt{-\lambda}\},
		\end{align}
		where the constant C depends on $\|\f{1+\theta^2}{\theta} + \f{1+(u')^2}{u'}\|_{L^\infty(-1,1)} $.
	\end{prop}

	\bibliographystyle{is-abbrv}
	\bibliography{reference}
	
\end{document}